\documentclass[12pt, twoside]{article}
\usepackage{amsmath,amssymb,amsfonts,amsthm,mathrsfs}
\usepackage{float}
\usepackage{graphicx}
\usepackage{epsfig}
\usepackage{cite}
\usepackage{indentfirst}
\usepackage{mdwlist}
\usepackage{enumitem}
\usepackage{epstopdf}
\usepackage{caption}
\usepackage{subfigure}
\usepackage{booktabs}
\usepackage{cases}
\usepackage[numbers]{natbib}
\usepackage[top=1in,bottom=1in,left=1.25in,right=1.25in]{geometry}
\usepackage[colorlinks,linkcolor=red,anchorcolor=green,citecolor=blue,CJKbookmarks=True]{hyperref}
\usepackage{caption}
\newtheorem{theorem}{Theorem}[section]
\newtheorem{lemma}{Lemma}[section]
\newtheorem{remark}{Remark}[section]
\newtheorem{proposition}{Proposition}[section]
\newtheorem{definition}{Definition}[section]
\newtheorem{corollary}{Corollary}[section]
\numberwithin{equation}{section} \makeatletter
\renewenvironment{proof}[1][\proofname]{\par
	\noindent\pushQED{\qed}%
	\normalfont \textbf{#1.} \hspace{0.5em}\ignorespaces
}{\popQED\par}
\begin{document}
\captionsetup[figure]{labelfont={bf},labelformat={default},labelsep=period,name={Fig.}}	
	\title{\textbf{Threshold dynamics of a time-periodic nonlocal dispersal SIS epidemic model with saturated incidence function and Neumann boundary conditions}}
	\author{Ziwei Liang, Xiandong Lin and Qiru Wang\thanks{Corresponding author. E-mail: mcswqr@mail.sysu.edu.cn}\\
		School of Mathematics, Sun Yat-sen University,\\ Guangzhou 510275, Guangdong, PR China\\
		}
	\date{}
	\maketitle
		
\begin{abstract}
In this paper, we consider a time-periodic nonlocal dispersal susceptible-infected-susceptible (SIS) epidemic model with saturated incidence and Neumann boundary conditions in a spatiotemporally heterogeneous environment. First, we define the basic reproduction number for the model, which depends on dispersal rates, total population size and saturation parameters, and is quite different from those of models incorporating standard or bilinear incidence. Then, we establish its variational characterization and investigate the impacts of those parameters on it. Next, we explore the existence, uniqueness and global attractivity of the equilibria. We also analyze the asymptotic behaviors of endemic steady states for small saturation parameters, and for both small and large diffusion rates. Our results show that the saturation effect enables the total population size to have a significant impact on the disease dynamics. Furthermore, as the saturation parameter tends to zero, the basic reproduction number and the endemic equilibrium reduce to those of the standard-incidence model, respectively. Finally, we support our findings by numerical simulations.
		
\textbf{Key words:} A nonlocal dispersal SIS epidemic model; time-periodic; saturated incidence function; basic reproduction number; threshold dynamics.

\textbf{2020 MSC:} 35K57; 35R20; 35B40; 92D25.
\end{abstract}
	
\section{Introduction}

There is a general consensus that spatial diffusion and environmental heterogeneity are important factors to be accounted for in the transmission of infectious diseases. Since the seminal work of Kermack and McKendrick \cite{Kermack1927}, numerous diffusive epidemic models have been developed to study the spatial spread of infectious diseases \cite{Bartlett1956,Bailey1980}. It is remarkable that earlier models ignore the movements of individuals and the coefficients are constants. In this context, Allen et al. \cite{Allen2008} proposed an SIS reaction-diffusion model with standard incidence
\begin{equation}
	\begin{cases}
		\frac{\partial S}{\partial t} = d_S \Delta S - \frac{\beta(x)SI}{S+I} + \gamma(x)I, & x \in \Omega, t > 0, \\
		\frac{\partial I}{\partial t} = d_I \Delta I + \frac{\beta(x)SI}{S+I} - \gamma(x)I, & x \in \Omega, t > 0, \\
		\partial_\nu S = \partial_\nu I = 0, & x \in \partial\Omega, t > 0.
	\end{cases}\label{model1.1}
\end{equation}
They explored the effects of spatial heterogeneity of the environment and the movement of individuals on the persistence and extinction of a disease. Subsequently, Peng and Liu \cite{Peng2009} extended the results in \cite{Allen2008} and established the global stability of the steady states. In \cite{R.Peng2009} and \cite{Peng2013}, the authors further investigated the impact of diffusion rates on disease transmission. Moreover, multifarious factors have been incorporated into the model (\ref{model1.1}) such as nonlocal dispersal, temporal periodicity, advective environments, logistic source, different incidence functions and boundary conditions, etc. \cite{Peng2012,Cui2016,Cui2017,Yang2017,Yang2019,Feng2022,Feng2026,Feng2025,Wang2026}. In recent years, different incidence functions, nonlocal dispersal, and temporal periodicity have received considerable research attention.

In epidemiological modeling, the incidence function plays a vital role in describing the rate of new infections and understanding the transmission dynamics of infectious diseases. Commonly used forms include bilinear, standard, and saturated incidence functions. Among these, the bilinear incidence term of the form $\beta SI$ is widely applied in the classical compartmental models such as the SIR and SIS models proposed by Kermack and McKendrick \cite{Kermack1927}. This form assumes that the contact rate is proportional to the total population size, without an upper bound. It is suitable for describing disease transmission in small populations, in the absence of intervention measures, and during the early stage of an outbreak. de Jong et al. \cite{deJong1995} proposed the standard incidence of the form $\frac{\beta SI}{S+I}$, which assumes a constant contact rate, i.e., frequency-dependent transmission independent of population density. This formulation is
particularly applicable to transmission within medium-to-large populations in the absence of significant behavioral interventions. The saturated incidence function $\frac{\beta SI}{m+S+I}$ developed by Diekmann and Kretzschmar \cite{Diekmann1991} captures the saturation effect, reflecting that the number of effective contacts an individual can have with others is bounded above due to the spatial or social distribution of the population and time constraints. It can be used to characterize disease transmission in large populations under intervention measures during epidemic peaks. Recently, growing evidence indicates that the modified frequency-dependent incidence function $\frac{\beta SI}{m+S+I}$ is appropriate for describing the transmission of infectious diseases under specific scenarios. For instance, Guo et al. \cite{Guo2022} and Gao et al. \cite{Gao2024} constructed diffusive SIS epidemic models with saturated incidence in a heterogeneous environment and investigated the threshold dynamics and asymptotic profiles of endemic steady states. For other types of incidence functions, including those with more complex nonlinearities, one may refer to \cite{V.Capasso1978,Anderson1979,Liu1986,Heesterbeek1993,Peng2021}.

With the emergence of multiscale transportation networks, the global distribution patterns of humans and organisms that travel with them have been fundamentally altered. This transformation has exerted a profound impact on the spread of epidemics: human infectious diseases rarely remain restricted to small geographical areas, but spread rapidly across countries and continents via the movement of infected individuals \cite{Hallatscheka2014}. For long-range dispersal of species including humans, reaction-diffusion operators are inadequate to accurately characterize disease transmission mechanisms. Consequently, nonlocal convolution operators, which can better capture long-range dispersal behavior, have been naturally introduced \cite{Fife2003,Brockmann2006}. The research on nonlocal epidemic models stems from the seminal work of Kendall \cite{Kendall1965,Kendall1977}, who generalized the classical Kermack-McKendrick model to a spatially dependent integro-differential equation. He utilized the integral term $\beta S(x,t) \int_{-\infty}^{+\infty}K(x-y)I(y,t)\, dy$ to describe how infected individuals $I(y,t)$ at location $y$ disperse to infect susceptible individuals $S(x,t)$ at location $x$. Recently,  Yang et al. \cite{Yang2019} formulated a nonlocal dispersal SIS epidemic model with standard incidence and Neumann boundary conditions. They explored the properties of the basic reproduction number and established the existence, uniqueness, and stability of steady states, along with the asymptotic profiles of endemic steady states for large dispersal rates. Feng et al. \cite{Feng2022} investigated the threshold dynamics of a nonlocal dispersal SIS epidemic model with bilinear incidence and explored the asymptotic profiles of positive steady states. Furthermore, Feng et al. \cite{Feng2026} proposed a nonlocal dispersal SIS epidemic model with saturated incidence. They analyzed the basic reproduction number and the limiting behaviors of positive steady states for both small and large dispersal rates. Further results on nonlocal epidemic models can be found in \cite{Yang2017,RassRadcliffe2003,Ruan2007,Kuniya2018,Zhao2026} and the references cited therein.

It should be pointed out that the transmission and recovery rates in the above models depend solely on spatial variables. In practice, such parameters usually display spatiotemporal heterogeneity. Typically, they vary periodically in time, driven by seasonal climatic cycles or socio-economic factors. Following this perspective, Peng and Zhao \cite{Peng2012} proposed a time-periodic SIS reaction-diffusion model. They defined the basic reproduction number $R_{0}$ of the model, upon which threshold results concerning the global dynamics were derived. Furthermore, the results demonstrate that the interplay of spatial heterogeneity and temporal periodicity tends to enhance the persistence of the disease. Based on \cite{Peng2012} and \cite{Yang2019}, Lin and Wang \cite{Lin2023,Lin2024} further incorporated factors such as nonlocal dispersal and temporal periodicity of the dispersal kernel, and proposed two time-periodic nonlocal dispersal SIS epidemic models with Neumann and Dirichlet boundary conditions, respectively. They investigated the threshold dynamics of the models and discussed the effects of large and small dispersal rates of susceptible and infected individuals on disease persistence and extinction. Feng et al. \cite{Feng2025} established the existence and asymptotic profiles of the endemic periodic solution for the nonlocal dispersal SIS epidemic model with standard incidence and logistic source in a time-periodic environment. Wang et al. \cite{Wang2026} studied the spatiotemporal dynamics of a periodic SIS epidemic model with standard incidence and external supply governed by Fokker-Planck-type diffusion law in a spatially heterogeneous environment. They found that periodicity promotes disease persistence and introduces increased complexity into the disease dynamics. Chen et al. \cite{Chen2026} formulated a reaction-diffusion-advection SIS epidemic model with saturated incidence in a spatiotemporally heterogeneous environment and obtained the spatial distribution of the disease in the case of sufficiently small diffusion of the infected population. More results on time-periodic epidemic models are available in \cite{Gao2024,Shaman2011,Bai2020,Zhang2021}.

Inspired by the references \cite{Yang2019,Feng2026,Feng2025,Lin2023} and others, in this paper, we consider the following nonlocal dispersal SIS epidemic model with saturated incidence function in a spatiotemporally heterogeneous environment
\begin{equation}
	\begin{cases}
		\displaystyle
		\frac{\partial S}{\partial t}
		= d_S \int_{\Omega} J(x-y,t)\big[S(y,t)-S(x,t)\big]\,dy
		- \frac{\beta(x,t) S I}{m(x,t)+S+I}
		+ \gamma(x,t) I, & x \in \Omega,\ t>0, \\[1.2em]
		
		\displaystyle
		\frac{\partial I}{\partial t}
		= d_I \int_{\Omega} J(x-y,t)\big[I(y,t)-I(x,t)\big]\,dy
		+ \frac{\beta(x,t) S I}{m(x,t)+S+I}
		- \gamma(x,t) I, & x \in \Omega,\ t>0, \\[1.2em]
		
		S(x,0)=S_0(x), \quad I(x,0)=I_0(x), & x \in \Omega,
	\end{cases}\label{model1.2}
\end{equation}
where the habitat $\Omega \subset \mathbb{R}^n$ $(n \ge 1)$ is a bounded domain with smooth boundary; $S(x,t)$ and $I(x,t)$ denote the densities of susceptible and infected individuals at location $x$ and time $t$, respectively; $d_S$ and $d_I$ are positive diffusion coefficients for susceptible and infected individuals, respectively; $\beta(x,t)$ and $\gamma(x,t)$ stand for the transmission rate and recovery rate at location $x$ and time $t$, respectively; $m(x,t)$ reflects the limitation on the effective number of contacts an individual can have with others due to population distribution and time constraints \cite{Gao2024}. The integral operator $\int_{\mathbb{R}^n} J(x-y,t)\big(u(y,t)-u(x,t)\big)dy$ describes diffusion processes, $\int_{\mathbb{R}^n} J(x-y,t)u(y,t)dy$ is the rate at which individuals are arriving at position $x$ from all other places, and
$-\int_{\mathbb{R}^n} J(x-y,t)u(x,t)dy$ represents the rate at which they are leaving location $x$ to travel to all other sites. Since the integrals are taken over the domain $\Omega$, it means that diffusion takes place only in $\Omega$. Individuals fail to enter or leave the domain $\Omega$. This is analogous to the homogeneous Neumann boundary condition. Therefore, we still call it a Neumann boundary condition, see e.g. \cite{Andreu2010}. Throughout this paper, we assume that
\begin{itemize}
	\item[(A1)] $J(x,t)$ is a nonnegative and continuous function on $\mathbb{R}^n \times \mathbb{R}$, and satisfies $J(x,t) = J(x,t+T)$, $J(x,t) = J(-x,t)$,  $J(0,t) > 0$  for all  $x\in \mathbb{R}^n$, $t\in \mathbb{R}$, $\int_{\mathbb{R}^n} J(x,t)\,dx = 1$ for all  $t\in \mathbb{R}$.
	\item[(A2)] $\beta,\gamma,m:\overline{\Omega}\times\mathbb{R}\to\mathbb{R}$ are continuous functions, all $T$-periodic in the second variable. Moreover, $\beta$ and $\gamma$ are strictly positive, and $m$ is nonnegative on $\overline{\Omega}\times\mathbb{R}$.
	\item[(A3)] $S_0(x)$ and $I_0(x)$ are nonnegative continuous functions on $\bar{\Omega}$, and the total number of initial infected individuals is positive, that is, $\int_{\Omega} I_0(x)\,dx > 0$.
\end{itemize}

It is widely recognized that the basic reproduction number $R_{0}$ serves as a threshold for determining the dynamics of epidemic models. For some  reaction-diffusion models (see, e.g., \cite{Allen2008,Peng2012}), this value equals the principal eigenvalue of a related eigenvalue problem. However, the nonlocal dispersal eigenvalue problem generally admits no principal eigenvalue, which can yield essential differences between nonlocal dispersal and reaction-diffusion problems. To address this challenge, we apply the theory developed by Thieme \cite{Thieme2009} to define the basic reproduction number of system (\ref{model1.2}). It should be pointed out that the basic reproduction number of the model with standard incidence \cite{Allen2008,Yang2019,Lin2023,Lin2024} depends only on $d_{I}$. In contrast to these models, the basic reproduction number of system (\ref{model1.2}) depends on $d_{I}$ as well as on the total population size and the saturation parameter $m(x,t)$. Therefore, we adopt the method proposed in \cite{arlin2026,arlin2026p} to derive a variational characterization of the basic reproduction number for system \eqref{model1.2}. Based on this characterization, we further explore the asymptotic behaviors of $R_0$ as the total population size, the saturation parameter $m(x,t)$, and the ratio of $m(x,t)$ to the total population size tend to zero and infinity individually. Compared with the model with saturated incidence in \cite{Feng2026}, system (\ref{model1.2}) is time-periodic. Therefore, the method adopted in \cite{Feng2026} for analyzing the limiting properties of the endemic equilibrium cannot be applied to our model. Inspired by \cite{Lin2023} and \cite{Lin2024}, we utilize the sub-super solutions method to obtain the limiting results of the endemic equilibrium of system (\ref{model1.2}) with respect to small saturation parameters and large/small diffusion coefficients when $d_S=d_I$. It is worth noting that when investigating the asymptotic behavior of the endemic equilibrium for large dispersal rates, we can relax the stringent condition $\beta(x,t)>\gamma(x,t)$ for all $(x,t)\in\overline{\Omega}\times\mathbb{R}$, imposed in~\cite{Lin2023}, to the weaker condition that $\int_{\Omega}\int_0^T \beta(x,t)\,dt\,dx>\int_{\Omega}\int_0^T \gamma(x,t)\,dt\,dx$. Correspondingly, this makes the analysis concerning the limiting behavior of the endemic equilibrium more delicate. In addition, the long-time behavior of the model is not taken into account in \cite{Feng2026}, which we supplement in our work.

Overall, the introduction of saturated incidence renders both the total population size and the saturation parameter to exert a substantial influence on disease dynamics. Specifically, the saturation effect diminishes transmission risk and can markedly reshape the spatial distribution of the disease under specific conditions. As for the role of total population siz, our findings indicate that whether restricting the mobility of susceptible and infected individuals can eradicate the disease relies not merely on the sign of the difference between the period-averaged transmission rate and the period-averaged recovery rate but also on the total population size, a feature that distinguishes our model from those with standard or bilinear incidence functions. Furthermore, as the saturation parameter tends to zero, the basic reproduction number and the endemic equilibrium reduce to those of the standard-incidence model, respectively.

The remainder of the paper is organized as follows. In Section \ref{section 2}, we define the basic reproduction number of system (\ref{model1.2}), derive its variational characterization and investigate its limiting profiles. In Section \ref{section 3}, we establish the existence, uniqueness and global attractivity of the disease-free equilibrium and the endemic equilibrium. Section \ref{section 4} is devoted to exploring the asymptotic profiles of the endemic equilibrium of system (\ref{model1.2}) for small saturation parameters, and for both large and small diffusion rates. In Section \ref{section 5}, numerical simulations are carried out to verify our theoretical findings. Finally, we summarize our paper with some discussions.

\section{The basic reproduction number and its variational characterization}
	\label{section 2}

In this section, we define the basic reproduction number for system (\ref{model1.2}) by using the standard theory from \cite{Wang2008,Thieme2009}, and then investigate its dependence on the diffusion rate \(d_I\), the saturation parameter \(m(x,t)\), the total population size, and the ratio of \(m(x,t)\) to the total population size.

Following the method developed in \cite{arlin2026,arlin2026p}, we establish a variational characterization of the basic reproduction number and then apply it to study the limiting profile of $R_{0}$. This provides an alternative approach to examining the effects of parameters on the basic reproduction number. In contrast, the method of Zhang and Zhao \cite{Zhang2021} transforms the problem concerning the limiting profile into the study of the spectral bound of the corresponding nonlocal dispersal operator, whereas our approach relies on the variational characterization established above.

Let \( X := C(\bar{\Omega}) \) be an ordered Banach space, equipped with the maximum norm, and a positive cone \( X_+ := \{u \in X : u(x) \geq 0, \forall x \in \bar{\Omega}\} \). Define
\[
C_T := \{u \in C(\bar{\Omega} \times \mathbb{R}) : u(x, t) = u(x, t + T), \forall (x, t) \in \bar{\Omega} \times \mathbb{R}\},
\]
and
\[
C_T^+ := \{u \in C_T : u(x, t) \geq 0, \forall (x, t) \in \bar{\Omega} \times \mathbb{R}\}.
\]
For a closed linear operator \( A \) on \(X\), 	define its spectral bound by
	\[
	s(A) = \sup\{\operatorname{Re} \lambda \in \mathbb{R} : \lambda \in \sigma(A)\},
	\]
	where \( \sigma(A) \) denotes the spectrum of \( A \). The spectral radius of \( A \) is defined by
	\[
	r(A) = \sup\{|\lambda| : \lambda \in \sigma(A)\}.
	\]
For convenience, let
\[
\mathcal{J}[u](x,t) := \int_{\Omega} J(x - y, t) [u(y, t) - u(x, t)] \, dy, \quad \forall u \in C(\bar{\Omega} \times \mathbb{R}).
\]
Adding the first and second equations of system (\ref{model1.2}) gives
\[
\frac{\partial}{\partial t} (S(x,t) + I(x,t)) = d_S \mathcal{J} [S] + d_I \mathcal{J} [I], \quad t \geq 0.
\]
Integrating it over $\Omega$, it follows from the symmetry of $J(x-y,t)$ that
\[
\frac{\partial}{\partial t} \int_{\Omega} (S(x, t) + I(x, t)) \, dx = 0, \quad t \geq 0.
\]
This indicates that the total population size is constant, denoted by $N$, that is
\[
\int_{\Omega} (S(x, t) + I(x, t)) \, dx = N, \quad t \geq 0.
\]

Linearizing system (\ref{model1.2}) at the unique disease-free equilibrium $(\frac{N}{\mid\Omega\mid}, 0)$, whose existence and uniqueness will be established in Section \ref{section 3} below, one can obtain the following decoupled equation for $I$
\begin{equation}
\frac{\partial u(x,t)}{\partial t} = d_I \mathcal{J} [u](x, t) + \theta(x, t) u(x, t) - \gamma(x, t) u(x, t), \quad x \in \Omega, \, t > 0,
\label{2.1}
\end{equation}
where $\theta(x, t)=\frac{\beta(x,t)\frac{N}{\mid\Omega\mid}}{m(x,t)+\frac{N}{\mid\Omega\mid}}$. Let \(\{V(t, s): t \geq s\}\) be the evolution family on \(X\) defined by following equation
\[
\frac{\partial u}{\partial t} = d_I \mathcal{J} [u] - \gamma(x, t) u, \quad x \in \Omega, \, t > 0.
\]
Let \( \phi(x, s) \) denote the density distribution of infected individuals at location \( x \in \Omega \) and time \( s \). Then \( \theta(x, s)\phi(x, s) \) corresponds to the density distribution of secondary infections generated by cases introduced at time \( s \). Moreover, for any fixed \( t \geq s \), \( V(t, s)\theta(x, s)\phi(x, s) \) is the density distribution at location \( x \) of those infected individuals who were newly infected at time \( s \) and remain infected at time \( t \). Thus, the density distribution of cumulative new infections at location \( x \) and time \( t \) generated by all infected individuals \( \phi(x, s) \) introduced at all times prior to \( t \) is
\begin{equation*}
	 \int_{0}^{\infty} V(t, t-s)\theta(\cdot, t-s)\phi(\cdot, t-s)\,ds.
\end{equation*}
Define
\begin{equation*}
	\mathcal{L}[\phi](x, t) := \int_{0}^{\infty} V(t, t-s)\theta(x, t-s)\phi(x, t-s)\,ds, \quad \phi \in C_T.
	\label{eq:Loperator}
\end{equation*}
To reflect the dependence of the basic reproduction number on $d_I$, $N$, and $m(x,t)$, the basic reproduction number for system (\ref{model1.2}) is defined as
\begin{equation*}
	R_0 = R_0( d_I, N, m) = r(\mathcal{L}),
\end{equation*}
where \( r(\mathcal{L}) \) stands for the spectral radius of \( \mathcal{L} \). Define
\begin{align*}
	&\mathcal{B}[u](x,t) := -\partial_t u(x,t) + d_I\mathcal{J}[u](x,t) - \gamma(x,t)u(x,t), \quad u\in C_T,\\[0.2em]
	&\mathcal{F}[u](x,t) := \theta(x,t)u(x,t), \quad u\in C_T.
\end{align*}
According to \citep[Proposition 2.3]{Feng2025}, we have $\mathcal{B}$ is a resolvent positive operator on $C_T$ and $s(\mathcal{B}) < 0$. For any $\mu>0$, set
$
	\mathcal{M}_{\mu}:= \mathcal{B}+\frac{1}{\mu}\mathcal{F}.
$

We first present some results on the basic reproduction number, which will be useful in analyzing the limiting behavior of $R_{0}$. Following the same arguments as in \citep[Lemma 2.5 or Theorem 2.10]{Zhang2021}, \citep[Lemma 2.2 and Proposition 2.2]{Lin2023} or \citep[Proposition 2.4]{Feng2025}, we have the following lemma.

\begin{lemma}
	\label{lemma2.3}
	The following assertions are valid.
	\begin{itemize}
		\item[{\rm (i)}] For any $\mu>0$, \( R_0 - \mu \) has the same sign as \( s(\mathcal{M}_{\mu})=s\left(\mathcal{B} + \frac{1}{\mu}\mathcal{F}\right) \){\rm;}
		\item[{\rm (ii)}] \( \mu = R_0 >0\) is the unique positive solution of $s(\mathcal{M}_{\mu})=0$.	
	\end{itemize}
\end{lemma}

Now we are in a position to give the main result of this section.

\begin{theorem}
	\label{theorem2.2}
	The following statements hold.
	\begin{itemize}
		\item[{\rm (i)}]\quad
		\(\begin{aligned}[t]
			R_0 &= \inf_{\varphi \in \mathcal{C}}\sup_{(x,t)\in\bar{\Omega}\times[0,T]} \frac{\theta(x,t)\varphi(x,t)}{\partial_t \varphi(x,t) - d_I\mathcal{J}[\varphi](x,t) + \gamma(x,t)\varphi(x,t)} \\
			&= \sup_{\varphi \in \operatorname{Int} C_T^+ \cap C^{0,1}\big(\bar{\Omega} \times \mathbb{R}\big)} \inf_{(x,t)\in\Omega_{\varphi}} \frac{\theta(x,t)\varphi(x,t)}{\partial_t \varphi(x,t) - d_I\mathcal{J}[\varphi](x,t) + \gamma(x,t)\varphi(x,t)},
		\end{aligned}\)

       where
\[ \mathcal{C}:= \left\lbrace \varphi \in \operatorname{Int} C_T^+ \cap C^{0,1}\big(\bar{\Omega} \times \mathbb{R}\big) : \mathcal{B}[\varphi](x,t)<  0,  \forall  (x,t)\in  \bar{\Omega} \times [0,T] \right\rbrace \]
and
$\Omega_{\varphi}:= \left\lbrace (x,t)\in  \bar{\Omega} \times [0,T]  : \mathcal{B}[\varphi](x,t)<  0 \right\rbrace $ for each $\varphi \in \operatorname{Int} C_T^+ \cap C^{0,1}\big(\bar{\Omega} \times \mathbb{R}\big) ${\rm;}		
		\item[{\rm (ii)}]
		$\lim\limits_{d_I\to 0} R_0(d_{I}, N, m)
		= \max\limits_{x\in\bar{\Omega}} \dfrac{\int_0^T \theta(x,t)dt}{\int_0^T \gamma(x,t)dt}${\rm;}~~$\lim\limits_{d_I\to +\infty} R_0(d_{I}, N, m)
		= \dfrac{\int_{\Omega}\int_0^T \theta(x,t)dtdx}{\int_{\Omega}\int_0^T \gamma(x,t)dtdx}${\rm;}
		\item[{\rm (iii)}]
		If $m$ is strictly positive, then
		$\lim\limits_{N\to 0} R_0(d_{I}, N, m)
		= 0${\rm,}~~$\lim\limits_{N\to +\infty} R_0(d_{I}, N, m)
		= \hat{R}_0$,~where $\hat{R}_0=r(\hat{\mathcal{L}})$,~ $\hat{\mathcal{L}}$ is defined by replacing $\theta(x,t)$ by $\beta(x,t)$ in the definition of $\mathcal{L}${\rm;}
		\item[{\rm (iv)}]
		$\lim\limits_{m_{max}\to 0} R_0(d_{I}, N, m)
		= \hat{R}_0${\rm,}~~$\lim\limits_{m_{min}\to +\infty} R_0(d_{I}, N, m)
		= 0$, where~ $m_{max}=\max_{(x,t)\in\bar{\Omega}\times[0,T]}m(x,t)$, $m_{min}=\min_{(x,t)\in\bar{\Omega}\times[0,T]}m(x,t)${\rm;} 
		\item[{\rm (v)}]
		Define the sequence $P_{n}(x,t):=\frac{m_{n}(x,t)}{N_{n}}$. Then
		$\lim\limits_{P_{n}\to 0} R_0(d_{I}, N, m)
		= \hat{R}_0${\rm,}~~$\lim\limits_{P_{n}\to +\infty} R_0(d_{I}, N, m)
		= 0$, ~~$\lim\limits_{P_{n}\to P(x,t)} R_0(d_{I}, N, m)=\tilde{R}_0$,~where $P(x,t)$ is a nontrivial positive $T$-periodic function, $\tilde{R}_0=r(\tilde{\mathcal{L}})$,~ $\tilde{\mathcal{L}}$ is defined by replacing $\theta(x,t)$ by $\frac{\beta(x,t)}{|\Omega|P(x,t)+1}$ in the definition of $\mathcal{L}$.
	\end{itemize}
\end{theorem}

\begin{proof}
	(i) By a similar argument as in \cite{arlin2026,arlin2026p}, we can prove (i). Indeed, set
	\[
	\mu_1: = \inf_{\varphi \in\mathcal{C}}\sup_{(x,t)\in\bar{\Omega}\times[0,T]} \frac{\theta(x,t)\varphi(x,t)}{\partial_t \varphi(x,t) - d_I\mathcal{J}[\varphi](x,t) + \gamma(x,t)\varphi(x,t)},
	\]
	and
	\[
	\mu_2: = \sup_{\varphi \in \operatorname{Int} C_T^+ \cap C^{0,1}\big(\bar{\Omega} \times \mathbb{R}\big)} \inf_{(x,t)\in\Omega_{\varphi}} \frac{\theta(x,t)\varphi(x,t)}{\partial_t \varphi(x,t) - d_I\mathcal{J}[\varphi](x,t) + \gamma(x,t)\varphi(x,t)}.
	\]
	It is easy to see that \(\mu_1\) and \(\mu_2\) can be equivalently written as
	\[
	\mu_1 = \inf \{\mu > 0 : \exists \varphi \in \operatorname{Int} C_T^+ \cap C^{0,1}\big(\bar{\Omega} \times \mathbb{R}\big) \text{ s.t. } \mathcal{M}_{\mu} [\varphi] \leq 0\},
	\]
	and
	\[
	\mu_2 = \sup \{\mu > 0 : \exists \varphi \in \operatorname{Int} C_T^+ \cap C^{0,1}\big(\bar{\Omega} \times \mathbb{R}\big) \text{ s.t. } \mathcal{M}_{\mu} [\varphi] \geq 0\}.
	\]
	Define
	\begin{equation*}
		\lambda_p(\mathcal{M}_{\mu}): = \sup \left\{ \lambda \in \mathbb{R} : \exists \varphi \in \operatorname{Int} C_T^+ \cap C^{0,1}\big(\bar{\Omega} \times \mathbb{R}\big), \text{ s.t. } \mathcal{M}_{\mu}[\varphi]  \ge \lambda\varphi \right\},
	\end{equation*}
	and
	\begin{equation*}
		\lambda_p'(\mathcal{M}_{\mu}): = \inf \left\{ \lambda \in \mathbb{R} : \exists \varphi \in \operatorname{Int} C_T^+ \cap C^{0,1}\big(\bar{\Omega} \times \mathbb{R}\big), \text{ s.t. } \mathcal{M}_{\mu}[\varphi] \le\lambda\varphi \right\} .
	\end{equation*}
	In view of \citep[Proposition 2.2]{Lin2024}, we have $s(\mathcal{M}_{\mu})=\lambda_p(\mathcal{M}_{\mu})=\lambda_p'(\mathcal{M}_{\mu})$. Obviously, $s(\mathcal{M}_{\mu})$ is non-increasing with respect to $\mu>0$. Furthermore, it is direct from the definition of $R_{0}$ that $R_{0}>0$. Then according to Lemma \ref{lemma2.3} (ii), we obtain that $s(\mathcal{M}_{\mu})=0$ admits the unique positive solution $R_{0}$. Therefore, $s(\mathcal{M}_{\mu})<0$ for all $\mu > R_{0}$ and $s(\mathcal{M}_{\mu})>0$ for all $0<\mu < R_{0}$. By the above arguments, for any \(\mu > R_{0}\), there exists $\psi\in\operatorname{Int} C_T^+ \cap C^{0,1}\big(\bar{\Omega} \times \mathbb{R}\big)$ such that $\mathcal{M}_{\mu} [\psi] \leq 0$, which indicates $R_{0}\geq \mu_1$. If $R_{0}> \mu_1$, it follows from the definition of $\mu_1$ that there exist some \(\mu' \in (\mu_1, R_{0})\) and \(\xi \in \operatorname{Int} C_T^+ \cap C^{0,1}\big(\bar{\Omega} \times \mathbb{R}\big)\) satisfying \(\mathcal{M}_{\mu'} [\xi] \leq 0\).  This, together with the definition of $\lambda_p'(\mathcal{M}_{\mu})$ and  $s(\mathcal{M}_{\mu})=\lambda_p'(\mathcal{M}_{\mu})$, demonstrates \(s(\mathcal{M}_{\mu'}) \leq 0\), a contradiction. Consequently, \(R_{0} = \mu_1\). Analogously, one can prove that \(R_{0} = \mu_2\).
	
The conclusion of (ii)  is directly obtained from \citep[Theorem 2.4]{Lin2023} or \citep[Theorem 2.6 (iii)]{Feng2025}.
To establish  (iii), (iv) and (v), we first show that  the continuity of $R_0$ on $\theta \in \operatorname{Int} C_T^+$, i.e.,
\begin{equation}\label{eq1}
\lim_{\|\theta-\theta_0\|\to 0} R_0(\theta)=R_0(\theta_0),
\end{equation}
where we write $R_0(\theta)$ to emphasize its dependence on $\theta$.
Let $\{\theta_n\}$ be any sequence converging to $\theta_0$.
For any $\varepsilon>0$, there exist $\phi_{1}, \phi_{2} \in \operatorname{Int} C_T^+ \cap C^{0,1}\big(\bar{\Omega} \times \mathbb{R}\big)$ such that
\[
\partial_t \phi_{1}(x,t) - d_I\mathcal{J}[\phi_{1}](x,t) + \gamma(x,t)\phi_{1}(x,t)\ge \frac{1}{R_{0}(\theta_0)+\varepsilon}\theta_{0}(x,t)\phi_{1}(x,t),
\]
and
\[
\partial_t \phi_{2}(x,t) - d_I\mathcal{J}[\phi_{2}](x,t) + \gamma(x,t)\phi_{2}(x,t)\le \frac{1}{R_{0}(\theta_0)-\varepsilon}\theta_{0}(x,t)\phi_{2}(x,t).
\]
Since $\left\| \theta_{n}-\theta_0\right\| \to 0$, there exists $n_{1}>0$
such that for all $n\ge n_{1}$,
\[
\partial_t \phi_{1}(x,t) - d_I\mathcal{J}[\phi_{1}](x,t) + \gamma(x,t)\phi_{1}(x,t)\ge \frac{1}{R_{0}(\theta_0)+2\varepsilon}\theta_{n}(x,t)\phi_{1}(x,t),
\]
and
\[
\partial_t \phi_{2}(x,t) - d_I\mathcal{J}[\phi_{2}](x,t) + \gamma(x,t)\phi_{2}(x,t)\le \frac{1}{R_{0}(\theta_0)-2\varepsilon}\theta_{n}(x,t)\phi_{2}(x,t).
\]
It follows that
\[
R_0(\theta_0)-2\varepsilon \le R_0(\theta_n) \le R_0(\theta_0)+2\varepsilon.
\]
By taking $\varepsilon\to0^+$, we obtain $\lim_{n\to\infty} R_0(\theta_n)=R_0(\theta_0)$. As $\{\theta_n\}$ is an arbitrary sequence, equality \eqref{eq1} holds. To finish the proof of assertions (iii), (iv) and (v), it remains to address the special case $\theta_n\to 0$ . In this situation, there exist constants $k_{n}$ such that $\left\| \theta_n \right\|\le k_{n} $ and $k_{n}\to 0$. It is easy to see that $R_{0}(\theta_n)\le R_{0}(k_{n})= k_{n}R_{0}(1)\to 0$.
This completes the proof.
\end{proof}

\begin{remark}
	Although Theorem {\rm\ref{theorem2.2} (i)} gives an expression for $R_0$, the monotonicity of $R_0$ in $d_{I}$ is not evident from this expression, and it turns out challenging to explore the monotonicity of $R_0$ with respect to $d_{I}$. For autonomous systems, Allen et al. {\rm\cite{Allen2008}} and Feng et al. {\rm\cite{Feng2026}} demonstrated that $R_0$ is a non-increasing function of $d_{I}$. However, for nonautonomous systems, Peng and Zhao {\rm\cite{Peng2012}} and Liu and Lou {\rm\cite{Liu2022}} pointed out that such monotonicity generally fails to hold. Recently, for time-periodic systems, Feng et al. {\rm\cite{Feng2025}} established the monotonicity of $R_0$ with respect to large dispersal rate $d_{I}$ by utilizing the expansions of the principal eigenvalue and the associated principal eigenfunction. Also, from Theorem {\rm\ref{theorem2.2} (iv)}, we see that as the saturation parameter tends to zero, the basic reproduction number reduces to that of the standard-incidence model.
\end{remark}

\begin{corollary}
	\label{corollary2.1}
	The following assertions hold.
	\begin{itemize}
		
		\item[{\rm (i)}]
		$\dfrac{\int_{\Omega}\int_0^T \theta(x,t)dtdx}{\int_{\Omega}\int_0^T \gamma(x,t)dtdx}\leq R_0(d_{I}, N, m)\leq\max\limits_{x(\cdot) \in \mathcal{S}} \dfrac{\int_0^T \theta(x(t), t) \, dt}{\int_0^T \gamma(x(t), t) \, dt}$,~~where \[\mathcal{S} := \{x(\cdot) \in C(\mathbb{R}, \bar{\Omega}) : x(t) = x(t + T)\}{\rm;}
		\]
		
		\item[{\rm (ii)}] If $\int_{\Omega}\int_{0}^{T} \left( \theta(x,t)-\gamma(x,t)\right) \,dt\,dx> 0$, then $R_0(d_{I}, N, m)>1$ for all $d_{I}>0$, $N>0$, and $m\in C_{T}^{+}${\rm;}
		
		\item[{\rm (iii)}] If  $\int_{0}^{T}\max\limits_{\substack{x\in\bar{\Omega}}} \left( \theta(x,t)-\gamma(x,t)\right)\,dt<0$, then $R_0(d_{I}, N, m)<1$ for all $d_{I}>0$, $N>0$, and $m\in C_{T}^{+}${\rm;}

		\item[{\rm (iv)}] For fixed $d_I$ and $m(x,t)>0$, $R_0(d_{I}, N, m)$ is a strictly increasing function of $N$;
		
		\item[{\rm (v)}] For fixed $d_I$ and $N$, $R_0(d_{I}, N, m)$ is strictly decreasing with respect to the pointwise ordering of $m$, that is to say, if
		\(m_1(x,t) <m_2(x,t)\) for all $(x,t)\in \bar{\Omega} \times \mathbb{R}$, then \(R_0(d_{I}, N, m_1) > R_0(d_{I}, N, m_2)\).
	\end{itemize}
\end{corollary}

\begin{proof} Statement (i) is a direct consequence of \citep[Proposition 2.5]{Feng2025}.  And statements (ii) and (iii) follow directly from  (i).
	
	(iv) It is easy to see that,  for any $N_{1}>N_{2}$, there exists $\varepsilon=\varepsilon(N_{1},N_{2})$ such that
	\[
	\dfrac{\frac{N_{1}}{\left| \Omega \right| } \beta(x,t)}{m(x,t)+\frac{N_{1}}{\left| \Omega \right| }} > (1+\varepsilon)\dfrac{\frac{N_{2}}{\left| \Omega \right| } \beta(x,t)}{m(x,t)+\frac{N_{2}}{\left| \Omega \right| }}.
	\]
	It follows from Theorem \ref{theorem2.2} (i) that $R_{0}(d_{I},N_{1},m)\ge(1+\varepsilon)R_{0}(d_{I},N_{2},m) $.
	
	Statement (v) can be proved analogously to (iv).
\end{proof}

Define
\begin{align*}
	H^+ &= \left\{ x \in \bar{\Omega} : \frac{1}{T}\int_0^T\beta(x,t)dt > \frac{1}{T}\int_0^T\gamma(x,t)dt \right\}, \\
	H^- &= \left\{ x \in \bar{\Omega} : \frac{1}{T}\int_0^T\beta(x,t)dt < \frac{1}{T}\int_0^T\gamma(x,t)dt \right\}.
\end{align*}
In line with Corollary \ref{corollary2.1} (iv) and Theorem \ref{theorem2.2} (iii), we derive the following corollary.

\begin{corollary}
	\label{corollary2.2}
	The following statements hold.
	\begin{itemize}
		\item[{\rm (i)}] If $H^-=\Omega$, then there exists $d_{1}^*>0$ such that $R_0(d_{I}, N, m)<1$ for all $d_{I}\in(0,d_{1}^*)$, all $N>0$, and all $m\in C_{T}^+${\rm;}
		\item[{\rm (ii)}]
		If $H^{+}$ is nonempty, then there exists $d_{2}^{*}>0$ such that for any fixed $d_{I}\in(0,d_{2}^{*})$ and $m\in \operatorname{Int} C_{T}^{+}$, there exists a threshold $N^{*}(d_{I},m)>0$ satisfying
\begin{equation}\label{eq:threshold}
	\begin{cases}
		R_0(d_I,N,m)<1, & 0<N<N^*(d_I,m),\\
		R_0(d_I,N,m)=1, & N=N^*(d_I,m),\\
		R_0(d_I,N,m)>1, & N>N^*(d_I,m).
	\end{cases}
\end{equation}
In particular, if $H^+=\Omega$, then for any fixed
$d_I>0$ and $m\in \operatorname{Int} C_{T}^{+}$, the threshold property of
$R_0(d_I,N,m)$ with respect to $N$ remains valid.
	\end{itemize}
\end{corollary}

\begin{proof}
(i) If $H^-=\Omega$, it follows from \citep[Theorem 2.6 (iii)]{Feng2025} that
\[
\lim\limits_{d_I\to 0} \hat{R}_0
= \max\limits_{x\in\bar{\Omega}} \dfrac{\int_0^T \beta(x,t)dt}{\int_0^T \gamma(x,t)dt}<1.
\]
Hence, there exists $d_{1}^*>0$ such that $\hat{R}_0<1$ for all $d_{I}\in(0,d_{1}^*)$. Since $\theta\le \beta$, by Theorem \ref{theorem2.2} (i), we  obtain that $R_0(d_{I}, N, m)\le\hat{R}_0 $ for all $d_{I}\in(0,d_{1}^*)$, $N>0$, and $m\in  C_{T}^{+}$. Statement (i) is thus proved. 

(ii) If $H^{+}$ is nonempty, we have  $\max\limits_{x\in\bar{\Omega}}\frac{1}{T}\int_0^T(\beta(x,t)-\gamma(x,t))dt>0$. In view of \citep[Theorem 2.6 (iii)]{Feng2025}, it follows that
\[
\lim\limits_{d_I\to 0} \hat{R}_0
= \max\limits_{x\in\bar{\Omega}} \dfrac{\int_0^T \beta(x,t)dt}{\int_0^T \gamma(x,t)dt}>1,
\]
which indicates that there exists $d_{2}^*>0$ such that $\hat{R}_0>1$ for all $d_{I}\in(0,d_{2}^*)$. Combining Corollary \ref{corollary2.1} (iv) and Theorem \ref{theorem2.2} (iii), for any
fixed $d_{I}\in(0,d_{2}^{*})$ and $m\in \operatorname{Int} C_{T}^{+}$, there exists a threshold $N^{*}(d_{I},m)>0$ satisfying (\ref{eq:threshold}). In particular, if $H^+=\Omega$, it is direct from \citep[Proposition 2.5]{Feng2025} that
\[
\hat{R}_{0}\geq\dfrac{\int_{\Omega}\int_0^T \beta(x,t)dtdx}{\int_{\Omega}\int_0^T \gamma(x,t)dtdx}>1.
\]
Then the desired result follows immediately from Corollary \ref{corollary2.1} (iv) and Theorem \ref{theorem2.2} (iii).
\end{proof}

\section{The existence, uniqueness and global attractivity of the equilibria}
	\label{section 3}
	
In this section, we explore the long-time behavior of system (\ref{model1.2}). It is evident that $\frac{SI}{m(x,t)+S+I}$ is a Lipschitz continuous function of $S$ and $I$ in the open first quadrant. We extend its definition to the entire first quadrant by setting it to be zero when $SI=0$. Applying the standard semigroup theory developed in \cite{Pazy1983} yields that, for any $(S_0, I_0) \in X_+ \times X_+$, system (\ref{model1.2}) admits a unique nonnegative solution $(S(x,t), I(x,t))$ for $(x,t) \in \Omega \times [0, T_{\max})$, where $T_{\max}$ denotes the maximal existence time.  Moreover, this solution is continuous in both $x$ and $t$. By replicating the argument of \cite{Lin2023}, one can show that $T_{\max}=+\infty$, and under assumption (A3), we also have $S(x, t)>0$ and $I(x, t)>0$. Therefore, we obtain the following proposition.

\begin{proposition}
	\label{proposition3.1}
	For any $(S_0, I_0) \in X_+ \times X_+$,  system {\rm(\ref{model1.2})} admits a unique  solution $(S(x, t), I(x, t))$ for $(x, t) \in \bar{\Omega} \times [0, +\infty)$, with $S(x, t)>0$ and $I(x, t)>0$.
\end{proposition}

In the following subsections, we consider the time-periodic problem of system (\ref{model1.2}) as follows
\begin{equation}
	\label{eq3.3}
	\begin{cases}
		\dfrac{\partial S}{\partial t} = d_S \mathcal{J} [S] - \dfrac{\beta(x,t)S I}{m(x,t)+S + I} + \gamma(x,t)I, & x \in \Omega, t \in \mathbb{R}, \\
		\dfrac{\partial I}{\partial t} = d_I \mathcal{J} [I] + \dfrac{\beta(x,t)S I}{m(x,t)+S + I} - \gamma(x,t)I, & x \in \Omega, t \in \mathbb{R}, \\
		S(x,t) = S(x,t + T), \quad I(x,t) = I(x,t + T), & x \in \Omega, t \in \mathbb{R}, \\
		\int_{\Omega} \bigl(S_0(x) + I_0(x)\bigr)\,dx = N.
	\end{cases}
\end{equation}

\begin{definition}
A solution $(S^*, I^*)$ of system {\rm(\ref{eq3.3})} is said to be globally attractive provided that for any initial data $(S_0, I_0) \in X_+ \times X_+$, it satisfies
\[
\lim_{t \to +\infty} \left\| (S(\cdot, t; S_0), I(\cdot, t; I_0)) - (S^*(\cdot, t), I^*(\cdot, t)) \right\|_{X \times X} = 0,
\]
where $(S(x, t; S_0), I(x, t; I_0))$ is the solution of system {\rm(\ref{model1.2})}.
\end{definition}

\subsection{The disease-free equilibrium}
In this subsection, we will give the existence for the solution $(S^{*}(x,t),I^{*}(x,t))$ of system (\ref{eq3.3}) with $S^{*}(x,t) > 0$ and  $I^{*}(x,t) = 0$, which is named as the disease-free equilibrium of system (\ref{model1.2}). Moreover, we establish the global attractivity of this equilibrium under two scenarios. Consider the following equation
\begin{equation}
	\label{eq3.4}
	\dfrac{\partial u}{\partial t} = d_I \mathcal{J} [u] + e_{1}(x, t) u - \gamma(x, t) u - e_{2}(x, t) u^2, ~~x \in \Omega, \; t > 0,
\end{equation}
where \( e_{1}(x, t) \) and \( e_{2}(x, t) \) are all continuous and time \( T \)-periodic functions with \( e_{2}(x, t) > 0 \) on \( \bar{\Omega} \times \mathbb{R} \). In line with \cite[Theorem 3.9]{Lin2026}, \citep[Theorem E]{Rawal2012} or \citep[Theorem B]{Shen2019}, we have the following conclusion.

\begin{lemma}
		\label{lemma3.2}
	The following statements hold.
	\begin{itemize}
		\item[{\rm (i)}]
		If $s(\mathcal{B}+\tilde{\mathcal{F}})>0$, then system {\rm(\ref{eq3.4})} admits a unique time periodic solution \( u^* \in \operatorname{Int} C_T^+ \). Furthermore, \( u^* \) is globally asymptotically stable, that is to say, for any \( u_0 \in X_+ \setminus \{0\} \),
		\begin{equation*}
			\|u(\cdot, t; u_0) - u^*(\cdot, t)\|_X \to 0 \quad \text{as } t \to +\infty,
		\end{equation*}
	where $\tilde{\mathcal{F}}$ is defined by $\tilde{\mathcal{F}}[u](x,t) := e_{1}(x,t)u(x,t)${\rm;}
		\item[{\rm (ii)}]
		If $s(\mathcal{B}+\tilde{\mathcal{F}})\le0$, then system {\rm(\ref{eq3.4})} admits no solution in $\mathcal{X}_{+}$, and there holds
		\begin{equation*}
			\|u(\cdot, t; u_0)\|_X \to 0 \quad \text{as } t \to +\infty,
		\end{equation*}
		where $\mathcal{X}_{+}$ denotes the set of all bounded and measurable functions  $u:   \Omega\times \mathbb{R}\to \mathbb{R}$ that satisfy $u(x,t)=u(x,t+T)$ and $u(x,t)\ge 0$ for all $(x,t)$.
	\end{itemize}
\end{lemma}

Following the same argument as in \cite[Proposition 3.2]{Lin2023}, we obtain the following conclusion.

\begin{proposition}
	\label{proposition3.2}
	System {\rm(\ref{model1.2})} admits a unique disease-free equilibrium \(\left(\frac{N}{|\Omega|}, 0\right)\).
\end{proposition}

Before proving the main results of this subsection, we recall a result that will play a key role in our proof.

\begin{lemma}{\rm (\citep[Proposition 2.4]{Lin2023})}.
	\label{lemma2.1}
	$\alpha(t)$ is continuous with respect to $t$, and there exists $\alpha_0 > 0$ such that
	\[
	\alpha_0 \leq \alpha(t) \leq \min_{x\in\bar{\Omega}} \int_{\Omega} J(x-y,t)\,dy,\quad t\in\mathbb{R}.
	\]
	Here
	\[
	\alpha(t) := \inf_{\substack{u\in L^2(\Omega),\,\int_{\Omega}u(x)dx=0,\,u\not\equiv 0}}
	\frac{\dfrac12 \displaystyle\int_{\Omega}\int_{\Omega} J(x-y,t)\big[u(y)-u(x)\big]^2 dxdy}
	{\displaystyle\int_{\Omega} u^2(x)dx},\quad t\in\mathbb{R}.
	\]
\end{lemma}

Now, we present the main results of this subsection.

\begin{theorem}
	\label{theorem3.1}
If $\beta(x,t)\leq\gamma(x,t)$ for all $(x, t) \in \bar{\Omega} \times [0, T]$, then the disease-free equilibrium \(\left(\frac{N}{|\Omega|}, 0\right)\) is globally attractive.
\end{theorem}

\begin{proof}
Let \((S(x,t), I(x,t))\) be the solution of system (\ref{model1.2}). Since $\beta(x,t)\leq\gamma(x,t)$, we have
\begin{equation*}
	\dfrac{\partial I}{\partial t} \leq d_I \mathcal{J} [I] + (\beta(x,t)-\gamma(x,t))I\leq d_I \mathcal{J} [I].
\end{equation*}
Consider the following auxiliary problem
\begin{equation}
	\label{eq3.5}
\begin{cases}
	\dfrac{\partial u}{\partial t} = d_I \mathcal{J} [u], & x \in \Omega, \; t > 0, \\
	u(x, 0) = I_0(x), & x \in \Omega.
\end{cases}
\end{equation}
It is commonly recognized that (\ref{eq3.5}) admits a unique solution $u(x,t)$. Moreover, it is obvious that $I$ is a sub-solution of (\ref{eq3.5}) and the constant \( \max\limits_{x \in \bar{\Omega}} I_0(x) \) is a super-solution of (\ref{eq3.5}). By applying the comparison principle, we get
 \[
 I(x, t) \leq u(x, t) \leq \max_{x \in \bar{\Omega}} I_0(x), \quad \forall (x, t) \in \bar{\Omega} \times [0, +\infty).
 \]
For the first equation of system (\ref{model1.2}), it is direct from the constant-variation formula that
\begin{equation}
	\label{eq3.6}
S(x, t) = e^{-\vartheta(x, t, 0)} S(x, 0) + \int_0^t e^{-\vartheta(x, t, s)} g(x, s) \, ds,
\end{equation}
where
\[
\vartheta(x, t, s) = d_S \int_s^t \int_{\Omega} J(x - y, r) \, dy \, dr,
\]
and
\[
g(x, t) = d_S \int_{\Omega} J(x - y, t) S(y, t) \, dy -\frac{\beta(x, t) SI}{m(x,t)+S + I}+\gamma(x,t)I.
\]
Notice that
\begin{equation}
		\label{eq3.7}
\begin{aligned}
	|g(x, t)|
	&\leq d_S \|J\|_{C(\mathbb{R}^n \times \mathbb{R})} \int_{\Omega} S(x, t) \, dx + \beta(x, t) I(x, t) +\gamma(x, t) I(x, t)\\
	&\leq d_S N \|J\|_{C(\mathbb{R}^n \times \mathbb{R})} + \|\beta+\gamma\|_{C_T} \max_{x \in \bar{\Omega}} I_0(x), \quad \forall (x, t) \in \bar{\Omega} \times \mathbb{R}.
\end{aligned}
\end{equation}
According to Lemma \ref{lemma2.1}, it follows that
\begin{equation*}
	\label{a}
	\int_{\Omega} J(x - y, t) \, dy \geq \alpha_0, \quad \forall x \in \bar{\Omega}, \; t \in \mathbb{R}.
\end{equation*}
Then $\vartheta(x, t, s)\geq d_S\alpha_0(t-s)$. Combining (\ref{eq3.6}) and (\ref{eq3.7}), we have
\[
\begin{aligned}
	|S(x, t)|
	&\leq e^{-d_S \alpha_0 t} S(x, 0) + \int_{0}^{t} e^{-d_S \alpha_0 (t-s)} |g(x, s)| \, ds \\
	&\leq \max_{x \in \bar{\Omega}} S_0(x) + \frac{d_S N \|J\|_{C(\mathbb{R}^n \times \mathbb{R})} + \|\beta+\gamma\|_{C_T} \max_{x \in \bar{\Omega}} I_0(x)}{d_S \alpha_0},
	\quad \forall (x, t) \in \bar{\Omega} \times \mathbb{R}_+.
\end{aligned}
\]
By the above argument, we obtain that there exists some positive constant $C_{0}$ independent of $t$ such that
\begin{equation}
	\label{eq3.8}
	\|S(\cdot, t)\|_X \leq C_0 \text{ and } \|I(\cdot, t)\|_X \leq C_0, \quad \forall t \geq 0.
\end{equation}
Now, we claim that
\[
I(x, t) \to 0 \text{ uniformly on } \bar{\Omega} \text{ as } t \to +\infty .
\]
According to assumption (A2), there exists a constant \( \beta_{1} > 0 \) such that $\beta(x, t) \geq \beta_1, ~\forall x \in \bar{\Omega}, \; t \in [0,T]$. By the $T$-periodicity of $\beta(x, t)$ in $t$, the above lower bound can be extended to all $t \in \mathbb{R}$, that is,
\[
\beta(x, t) \geq \beta_1, \quad \forall (x, t) \in \bar{\Omega} \times \mathbb{R}.
\]
Since $\|S+I\|_X\leq 2C_{0}$ and $\beta(x,t)\leq\gamma(x,t)$, it follows from the second equation of system (\ref{model1.2}) that
 \begin{align*}
\frac{\partial I}{\partial t} &\le d_I \mathcal{J} [I] +(\beta(x,t)-\gamma(x,t))I- \dfrac{\beta(x,t)I^{2}}{m(x,t)+S + I}\\
&\leq d_I \mathcal{J} [I]- \dfrac{\beta(x,t)I^{2}}{m(x,t)+S + I}\\
&\leq d_I \mathcal{J} [I]-\dfrac{\beta_{1}I^{2}}{\|m\|_{C_T}+2C_{0}}.
 \end{align*}
It is easy to see that $I(x,t)$ is a sub-solution of
the following initial value problem
\begin{equation}
	\begin{cases}
		\dfrac{{\rm d}\bar{I}(t)}{{\rm dt}} = -\dfrac{\beta_1 \bar{I}^2(t)}{\|m\|_{C_T}+2C_{0}}, & x \in \Omega, t > 0, \\
		\bar{I}(0) = \max\limits_{x \in \bar{\Omega}} I_0(x),&x \in \Omega.
	\end{cases}
	\label{eq3.9}
\end{equation}
Direct calculation shows that (\ref{eq3.9}) admits a unique solution $\bar{I}$, which is given by
\begin{equation*}
\bar{I}(t)=\frac{(2C_{0}+\|m\|_{C_T})\bar{I}(0)}{2C_{0}+\|m\|_{C_T}+\beta_{1}\bar{I}(0)t},~~\forall t\geq0.
\end{equation*}
Obviously, $\bar{I}(t) \to 0 ~\text{ as } ~t \to +\infty.$
Additionly, by the comparison principle, we can derive
 \[
I(x, t) \leq \bar{I}(t), \quad \forall (x, t) \in \bar{\Omega} \times [0, +\infty).
\]
To sum up,
\begin{equation}
	\label{eq3.10}
	\lim_{t \to +\infty} \|I(\cdot, t)\|_X = 0.
\end{equation}
Hence, $\int_{\Omega} S(x, t) dx \to N \text{ as } t \to +\infty.$ Let
\[
\tilde{S}(t) = \frac{1}{|\Omega|} \int_{\Omega} S(x, t) dx, \quad \text{and} \quad \hat{S}(x, t) = S(x, t) - \tilde{S}(t).
\]
A careful calculation gives that $\int_{\Omega} \hat{S}(x, t) dx=0$ and
\begin{equation}
	\label{eq3.11}
\frac{\partial \hat{S}}{\partial t} = d_S \mathcal{J} [\hat{S}] + h(x, t), \quad x \in \bar{\Omega}, t > 0,
\end{equation}
where
\[
h(x, t) = -\frac{\beta(x, t)SI}{m(x,t)+S + I}+\gamma(x,t)I + \frac{1}{|\Omega|} \int_{\Omega} \left(\frac{\beta(x, t)SI}{m(x,t)+S + I}-\gamma(x,t)I\right) dx.
\]
This together with (\ref{eq3.10}) demonstrates $\lim\limits_{t \to +\infty} \|h(\cdot, t)\|_X = 0$.
Let
\[
f(t) = \int_{\Omega} \hat{S}(x, t)h(x, t)dx.
\]
Noth that $\left| \hat{S}(x, t) \right| \leq 2C_0,~ \forall (x, t) \in \overline{\Omega} \times \mathbb{R}_+$, then
\[
|f(t)| \to 0 \text{ as } t \to +\infty.
\]
Define  $w(t) = \int_{\Omega} \hat{S}^2(x, t)dx$, direct calculation yields that
\begin{align*}
	\frac{dw}{dt}
	&= 2 \int_{\Omega} \hat{S}(x, t) \frac{\partial \hat{S}(x, t)}{\partial t} \, dx \nonumber \\
	&= 2d_S \int_{\Omega} \hat{S}(x, t) \left[ \int_{\Omega} J(x - y, t) \left( \hat{S}(y, t) - \hat{S}(x, t) \right) dy \right] dx + 2f(t) \nonumber \\
	&= -d_S \int_{\Omega} \int_{\Omega} J(x - y, t) \left[ \hat{S}(y, t) - \hat{S}(x, t) \right]^2 dx \, dy + 2 f(t)\\
	&\leq -2d_S \alpha_0 w(t) + 2f(t),
\end{align*}
where $\alpha_0$ is given in Lemma \ref{lemma2.1}. By the constant-variation formula and the comparison principle, we have
\[
w(t) \leq e^{-2d_S\alpha_0 t} w(0) + 2e^{-2d_S\alpha_0 t} \int_0^t e^{2d_S\alpha_0 s} f(s) ds,
\]
which indicates $w(t) \to 0$ as $t \to +\infty$.

Meanwhile, from (\ref{eq3.11}), we have
\begin{align*}
\left| \hat{S}(x, t) \right|
&= \left|
e^{-\vartheta(x, t, 0)} \hat{S}(x, 0)
+ \int_0^t e^{-\vartheta(x, t, s)}
\left[ d_S \int_{\Omega} J(x - y, s) \hat{S}(y, s) dy + h(x, s) \right] ds
\right|\\
&\leq e^{-d_S \alpha_0 t} \left| \hat{S}(x, 0) \right|
+ \int_0^t e^{-d_S \alpha_0 (t - s)}
\left[ d_S \int_{\Omega} J(x - y, s) \left| \hat{S}(y, s) \right| dy + \left| h(x, s) \right| \right] ds\\
&\leq e^{-d_S \alpha_0 t} \left| \hat{S}(x, 0) \right|
+ \int_0^t e^{-d_S \alpha_0 (t - s)}
\left[ C_1 w^{\frac{1}{2}}(s) + \left| h(x, s) \right| \right] ds
\end{align*}
for some positive constant $C_{1}$. It is worth mentioning that Hölder inequality was applied in the above estimate. Note that
\[
\lim_{t \to +\infty} \int_0^t e^{-d_S \alpha_0(t-s)} |h(x,s)| \, ds
= \lim_{t \to +\infty} \frac{|h(x,t)|}{d_S \alpha_0}
= 0 \quad \text{uniformly for } x \in \bar{\Omega},
\]
and
\[
\lim_{t \to +\infty} \int_0^t e^{-d_S \alpha_0(t-s)} w^{\frac{1}{2}}(s) \, ds = \lim_{t \to +\infty} \frac{w^{\frac{1}{2}}(t)}{d_S \alpha_0}= 0.
\]
Consequently, we obtain that
\[
\lim_{t \to +\infty} \left\| \hat{S}(\cdot, t) \right\|_X = 0.
\]
That is to say,
\[
\lim_{t \to +\infty} \left\| S(\cdot, t) - \frac{N}{|\Omega|} \right\|_X = 0.
\]
This ends the proof.
\end{proof}

\begin{theorem}
	\label{theorem3.2}
	Assume that $d_S=d_I=d$ and $R_0(d, N, m)\leq1$, then the disease-free equilibrium \(\left(\frac{N}{|\Omega|}, 0\right)\) is globally attractive.
\end{theorem}

\begin{proof}
As $d_S=d_I=d$, adding the two equations in system (\ref{model1.2}) gives
\begin{equation*}
	\begin{cases}
		\dfrac{\partial (S+I)}{\partial t} = d\mathcal{J} [S+I], & x \in \Omega, \; t > 0, \\
		(S+I)(x, 0) = S_0(x)+I_0(x), & x \in \Omega.
	\end{cases}
\end{equation*}
By \citep[Lemma 3.2]{Lin2023}, one can derive that
\begin{equation}
	S(x, t) + I(x, t) \to \frac{1}{|\Omega|}\int_{\Omega} (S_{0}(x)+I_{0}(x))\,dx=\frac{N}{|\Omega|} \text{ uniformly on } \bar{\Omega} \text{ as } t \to +\infty .
\end{equation}
Hence, for any small enough $\varepsilon > 0$, there is a large $t^{*}>0$ such that
\begin{equation}
	\label{eq3.13}
	\frac{N}{|\Omega|} - \varepsilon \leq S(x, t) + I(x, t) \leq \frac{N}{|\Omega|} + \varepsilon, \quad x \in \bar{\Omega}, \quad t \geq t^{*}.
\end{equation}
Notice that $I(x,t)$ solves
\begin{equation}
		\label{eq3.14}
	\begin{cases}
	\frac{\partial I}{\partial t} = d \mathcal{J} [I]+\left[\frac{\beta(x, t)(S+I)}{m(x,t)+S+I}-\gamma(x,t)-\frac{\beta(x,t)I}{m(x,t)+S+I}\right]I & x \in \Omega, \quad t > 0, \\
	I(x, 0) = I_{0}(x), & x \in \Omega.
\end{cases}
\end{equation}
In view of (\ref{eq3.13}) and (\ref{eq3.14}), we consider the following  auxiliary problem
\begin{equation}
	\label{eq3.15}
	\begin{cases}
		\dfrac{\partial \bar{I}_\varepsilon}{\partial t} = d \mathcal{J} [\bar{I}_\varepsilon] +\left[\frac{\beta(x, t)\left(\frac{N}{|\Omega|}+\varepsilon\right)}{m(x,t)+\frac{N}{|\Omega|}+\varepsilon}-\gamma(x,t)-\frac{\beta(x,t)\bar{I}_\varepsilon}{m(x,t)+\frac{N}{|\Omega|}+\varepsilon}\right]\bar{I}_\varepsilon, & x \in \Omega, \quad t > t^{*}, \\
		\bar{I}_\varepsilon(x, t^{*}) = I(x, t^{*}), & x \in \Omega.
	\end{cases}
\end{equation}
It is well known that system (\ref{eq3.15})  admits a unique solution, denoted by $\bar{I}_\varepsilon(x, t)$. Furthermore, the comparison principle implies that $\bar{I}_\varepsilon(x, t)$ is the super-solution of (\ref{eq3.14}) for $t\geq t^{*}$. As a result, we have
\begin{equation}
		\label{eq3.17}
	0 \leq I(x, t) \leq \bar{I}_\varepsilon(x, t), \quad x \in \bar{\Omega}, \quad t \geq t^{*}.
\end{equation}

We first discuss the case $R_0(d, N, m)<1$, which is equivalent to $s(\mathcal{B} + \mathcal{F})<0$. According to \cite[Proposition 2.4]{Lin2024}, we obtain $s(\mathcal{B}+\mathcal{F}_{+\varepsilon})>s(\mathcal{B} + \mathcal{F})$, where $\mathcal{F}_{+\varepsilon}$ is defined by
\[
\mathcal{F}_{+\varepsilon}[u](x,t) := \frac{\beta(x, t)\left(\frac{N}{|\Omega|}+\varepsilon\right)}{m(x,t)+\frac{N}{|\Omega|}+\varepsilon}u(x,t).
\]
Thanks to the continuous dependence of the spectral bound on $\varepsilon$, we may assume that $s(\mathcal{B}+\mathcal{F}_{+\varepsilon})<0$. By Lemma \ref{lemma3.2} (ii), one yields that
\[
I(x, t)\leq \bar{I}_\varepsilon(x, t)\to 0 \text{ uniformly on } \bar{\Omega} \text{ as } t \to +\infty .
\]

Now we discuss the case $R_0(d, N, m)=1$. In this case, $s(\mathcal{B} + \mathcal{F})=0$, then $s(\mathcal{B}+\mathcal{F}_{+\varepsilon})>0$ and $s(\mathcal{B}+\mathcal{F}_{+\varepsilon})\to 0$ as $\varepsilon\to 0$. It follows from Lemma \ref{lemma3.2} (i) that
\begin{equation}
	\label{eq3.18}
\lim_{t \to +\infty} \left\| \bar{I}_\varepsilon(\cdot, t) -\bar{I}^{*}_\varepsilon(\cdot, t) \right\|_X = 0,
\end{equation}
where $\bar{I}_\varepsilon^*\in \operatorname{Int}C_T^+$ is the unique positive periodic solution of
\[
\dfrac{\partial \bar{I}^{\ast}_\varepsilon}{\partial t} = d \mathcal{J} [\bar{I}^{\ast}_\varepsilon] +\left[\frac{\beta(x, t)\left(\frac{N}{|\Omega|}+\varepsilon\right)}{m(x,t)+\frac{N}{|\Omega|}+\varepsilon}-\gamma(x,t)-\frac{\beta(x,t)\bar{I}^{\ast}_\varepsilon}{m(x,t)+\frac{N}{|\Omega|}+\varepsilon}\right]\bar{I}^{\ast}_\varepsilon, \quad  x \in \Omega, \ t \in \mathbb{R}.
\]
It is direct from \citep[Proposition 3.3]{Lin2023} that $\bar{I}_\varepsilon^*$ is non-decreasing in $\varepsilon$. Therefore, there exists a sequence ${\varepsilon_{n}}$ with $\varepsilon_{n}\to 0$ as $n\to +\infty$ such that
\[
\bar{I}^*_{\varepsilon_{n}}(x,t)\to I_{1}(x,t) \text{ pointwise on } \bar{\Omega}\times \mathbb{R} \text{ as } n \to +\infty,
\]
where $I_{1}(x,t)$ satisfies
\begin{equation}
	\label{eq3.19}
	\begin{cases}
		\dfrac{\partial I_{1}(x,t)}{\partial t} = d \mathcal{J} [I_{1}] +\left[\frac{\beta(x, t)\frac{N}{|\Omega|}}{m(x,t)+\frac{N}{|\Omega|}}-\gamma(x,t)-\frac{\beta(x,t)I_{1}}{m(x,t)+\frac{N}{|\Omega|}}\right]I_{1}, & x \in \Omega, \quad t \in \mathbb{R}, \\
		I_{1}(x, t) = I_{1}(x, t+T), & x \in \Omega.
	\end{cases}
\end{equation}
According to \cite[Lemma 3.8]{Lin2026}, we have $I_{1}(x,t)\equiv0$.
By applying Dini's theorem, we get
\begin{equation}
	\label{eq3.21}
\bar{I}^*_{\varepsilon}(x,t)\to 0~ \text{in } C_{T} \text{ as } \varepsilon \to 0.
\end{equation}
Combining (\ref{eq3.17}), (\ref{eq3.18}) and (\ref{eq3.21}), we have
\[
I(x, t) \to 0 \text{ uniformly on } \bar{\Omega} \text{ as } t \to +\infty .
\]
Recall that
\[
S(x, t) + I(x, t) \to \frac{N}{|\Omega|} \text{ uniformly on } \bar{\Omega} \text{ as } t \to +\infty.
\]
Then
\[
\lim_{t \to +\infty} \left\| \left(S(\cdot, t),I(\cdot, t)\right) - \left(\frac{N}{|\Omega|}, 0\right) \right\|_{X\times X } = 0.
\]
This completes the proof.
\end{proof}

\subsection{The endemic equilibrium}
In this subsection, we will establish the existence and uniqueness of the positive solution of system (\ref{eq3.3}), which is called as the endemic equilibrium of system (\ref{model1.2}). Moreover, we will explore the long-time behavior of system (\ref{model1.2}) when $d_{S}=d_{I}$.

\begin{lemma}
	\label{lemma3.4}
	Suppose that \( d_I = d_S = d \). Then \( (S^*(x,t), I^*(x,t)) \) is a solution of {\rm(\ref{eq3.3})} if and only if it is a solution of
	\begin{equation}
		\label{eq3.22}
		\begin{cases}
			\frac{N}{|\Omega|} = S^*(x, t) + I^*(x, t), & x \in \Omega, t \in \mathbb{R}, \\
			\frac{\partial I^*}{\partial t} = d\mathcal{J}[I^*] +\left[\frac{\beta(x, t)\frac{N}{|\Omega|}}{m(x,t)+\frac{N}{|\Omega|}}-\gamma(x,t)-\frac{\beta(x,t)I^*}{m(x,t)+\frac{N}{|\Omega|}}\right]I^*, & x \in \Omega, t \in \mathbb{R}, \\
			I^*(x, t) = I^*(x, t + T), & x \in \Omega, t \in \mathbb{R}.
		\end{cases}
	\end{equation}
\end{lemma}

\begin{proof}
	Note that \( d_I = d_S = d \). If \( (S^*, I^*) \) is the solution of (\ref{eq3.3}), then we have
	\begin{equation*}
		\frac{\partial}{\partial t} (S^* + I^*) = d\mathcal{J}[S^* + I^*].
	\end{equation*}
By arguments similar to those in \cite[Proposition 3.2]{Lin2023}, we obtain that
	\begin{equation*}
		S^*(x, t) + I^*(x, t) = \frac{N}{|\Omega|} \text{ on } \bar{\Omega} \times \mathbb{R}.
	\end{equation*}
Substituting $S^*=\frac{N}{|\Omega|}-I^*$ into the second equation of (\ref{eq3.3}), we derive the equation about \( I^* \) in (\ref{eq3.22}).
	
On the other hand, if \( (S^*, I^*) \) is a solution of system (\ref{eq3.22}), it follows that
\begin{equation}
	\begin{aligned}
		\frac{\partial S^*}{\partial t}
		&= -d\mathcal{J} [I^*] - \left[\frac{\beta(x, t)\frac{N}{|\Omega|}}{m(x,t)+\frac{N}{|\Omega|}}-\gamma(x,t)-\frac{\beta(x,t)I^*}{m(x,t)+\frac{N}{|\Omega|}}\right]I^* \\
		&= d\mathcal{J} [S^*] - \frac{\beta(x, t)S^* I^*}{m(x,t)+S^* + I^*} + \gamma(x, t)I^*, \quad x \in \Omega, t \in \mathbb{R},
	\end{aligned}
\end{equation}
which suggests that \((S^*, I^*)\) satisfies (\ref{eq3.3}). The proof is complete.
\end{proof}

\begin{theorem}
	\label{theorem3.3}
	Suppose \( d_I = d_S = d \) and \( R_0(d, N, m) > 1 \). Then {\rm(\ref{eq3.3})} admits a unique positive solution \( (S_{e}^*, I_{e}^*) \in \operatorname{Int} C^+_T \times \operatorname{Int} C^+_T\).
\end{theorem}

\begin{proof}
Since \( R_0(d, N, m) > 1 \), which is equivalent to $s(\mathcal{B} + \mathcal{F})>0$ by Lemma \ref{lemma2.3} (i). By virtue of Lemma \ref{lemma3.2} (i), one can obtain that the second equation of (\ref{eq3.22}) admits a unique positive periodic solution, denoted by $I_{e}^* \in \mathrm{Int}C^+_T$. Note that \( d_I = d_S = d \). Then it follows from Lemma \ref{lemma3.4} that \( \left( \frac{N}{|\Omega|}-I_{e}^*, I_{e}^* \right) \) is the solution of (\ref{eq3.3}). In addition, it is obvious that
\[
\frac{\partial}{\partial t} \left( \frac{N}{|\Omega|} \right) > d\mathcal{J} \left[\frac{N}{|\Omega|}\right] +
\left[\frac{\beta(x, t)\frac{N}{|\Omega|}}{m(x,t)+\frac{N}{|\Omega|}}-\gamma(x,t)-\frac{\beta(x,t)\frac{N}{|\Omega|}}{m(x,t)+\frac{N}{|\Omega|}}\right]\frac{N}{|\Omega|}.
\]
Then it is direct from \citep[Proposition 3.3]{Lin2023} that $I_{e}^*(x, t) \leq \frac{N}{|\Omega|}$. However, when $I_{e}^*(x, t)=\frac{N}{|\Omega|}$, $I_{e}^*$ does not satisfy the second equation of (\ref{eq3.22}), yielding a contradiction. Therefore, $I_{e}^*(x, t)<\frac{N}{|\Omega|}$, which implies that $\frac{N}{|\Omega|}-I_{e}^*\in \mathrm{Int} C^+_T$. In conclusion, (\ref{eq3.3}) admits a unique positive solution \( (S_{e}^*, I_{e}^*) \in \operatorname{Int} C^+_T \times \operatorname{Int} C^+_T\), where $S_{e}^*=\frac{N}{|\Omega|}-I_{e}^*$. This completes the proof.
\end{proof}

\begin{theorem}
		\label{theorem3.4}
	Assume that \( d_I = d_S = d \) and $R_0(d,N,m) > 1$. Then the endemic equilibrium $(S_{e}^*, I_{e}^*)$ is globally attractive.
\end{theorem}

\begin{proof}
Since \( d_I = d_S = d \), as shown in the proof of Theorem \ref{theorem3.2}, (\ref{eq3.13}) and (\ref{eq3.14}) still hold. We consider the auxiliary problem (\ref{eq3.15}) and the following one
\begin{equation}
	\label{eq3.16}
	\begin{cases}
		\dfrac{\partial \underline{I}_\varepsilon}{\partial t} = d \mathcal{J} [\underline{I}_\varepsilon] +\left[\frac{\beta(x, t)\left(\frac{N}{|\Omega|}-\varepsilon\right)}{m(x,t)+\frac{N}{|\Omega|}-\varepsilon}-\gamma(x,t)-\frac{\beta(x,t)\underline{I}_\varepsilon}{m(x,t)+\frac{N}{|\Omega|}-\varepsilon}\right]\underline{I}_\varepsilon, & x \in \Omega, \quad t > t^{*}, \\
		\underline{I}_\varepsilon(x, t^{*}) = I(x, t^{*}), & x \in \Omega.
	\end{cases}
\end{equation}
Let $\underline{I}_\varepsilon(x,t)$ be the solution of \eqref{eq3.16}. By the comparison principle, we have
\begin{equation}
	\label{sub-super}
	\underline{I}_\varepsilon(x, t) \leq I(x, t) \leq \bar{I}_\varepsilon(x, t), \quad x \in \bar{\Omega}, \quad t \geq t^{*}.
\end{equation}	
If $R_0(d, N, m) > 1$, it follows from Lemma \ref{lemma2.3} (i) that $s(\mathcal{B} + \mathcal{F})>0$. In view of \cite[Proposition 2.4]{Lin2024}, we have $s(\mathcal{B}+\mathcal{F}_{-\varepsilon})<s(\mathcal{B} + \mathcal{F})<s(\mathcal{B}+\mathcal{F}_{+\varepsilon})$, where $\mathcal{F}_{+\varepsilon}$ is defined as in Theorem \ref{theorem3.2}, and $\mathcal{F}_{-\varepsilon}$ is defined by
\[
\mathcal{F}_{-\varepsilon}[u](x,t) := \frac{\beta(x, t)\left(\frac{N}{|\Omega|}-\varepsilon\right)}{m(x,t)+\frac{N}{|\Omega|}-\varepsilon}u(x,t).
\]
We may assume that $s(\mathcal{B}+\mathcal{F}_{\pm\varepsilon})>0$ for all small $\varepsilon$ due to the continuous dependence of the spectral bound on $\varepsilon$. By virtue of Lemma \ref{lemma3.2} (i), (\ref{eq3.18}) remains true and it also holds that
\begin{equation}
	\label{eq3.24}
	\lim_{t \to +\infty} \left\| \underline{I}_\varepsilon(\cdot, t) -\underline{I}^{*}_\varepsilon(\cdot, t) \right\|_X = 0,
\end{equation}	
where $\underline{I}^{*}_\varepsilon\in \operatorname{Int}C_T^+$ is the unique positive periodic solution of (\ref{eq3.16}). By \citep[Proposition 3.3]{Lin2023}, we have $\bar{I}_\varepsilon^*$ is non-decreasing in $\varepsilon$, and $\underline{I}^{*}_\varepsilon$ is non-increasing in $\varepsilon$, and $\bar{I}^{*}_{\varepsilon_{1}}\geq\underline{I}^{*}_{\varepsilon_{2}}$ for any $\varepsilon_{1}>0$ and $\varepsilon_{2}>0$.
Thus, we get
\begin{equation}
	\label{eq3.25}
\bar{I}^*_{\varepsilon}(x,t)\to I^{*}_{1}(x,t)~ \text{ pointwise on } \bar{\Omega}\times \mathbb{R} \text{ as } \varepsilon \to 0,
\end{equation}	
and
\begin{equation}
	\label{eq3.26}
	\underline{I}^*_{\varepsilon}(x,t)\to I^{*}_{2}(x,t)~ \text{ pointwise on } \bar{\Omega}\times \mathbb{R} \text{ as } \varepsilon \to 0,
\end{equation}	
where $I^{*}_{1}\in L^{\infty}$ and $I^{*}_{2}\in L^{\infty}$ satisfy (\ref{eq3.19}). Since $s(\mathcal{B} + \mathcal{F})>0$, according to Lemma \ref{lemma3.2} (i), we obtain that (\ref{eq3.19}) has a unique positive periodic solution. Consequently, $I^{*}_{1}=I^{*}_{2}\in \text{Int} C^{+}_{T}$.

Using  Dini's theorem and combining (\ref{eq3.18}) and (\ref{sub-super})-(\ref{eq3.26}), we have
\[
I(x, t) \to I^{*}_{1}(x,t) \text{ uniformly on } \bar{\Omega} \text{ as } t \to +\infty .
\]
Since
\[
S(x, t) + I(x, t) \to \frac{N}{|\Omega|}\text{ uniformly on } \bar{\Omega} \text{ as } t \to +\infty,
\]
it follows that
\[
S(x, t) \to \frac{N}{|\Omega|}-I^{*}_{1}(x,t) \text{ uniformly on } \bar{\Omega} \text{ as } t \to +\infty.
\]
By comparing (\ref{eq3.19}) with (\ref{eq3.22}), it follows from Lemma \ref{lemma3.4} that $I_{e}^*=I^{*}_{1}=I^{*}_{2}$ and $S_{e}^*=\frac{N}{|\Omega|}-I_{e}^*$. As a result,
\[
\lim_{t \to +\infty} \left\| \left(S(\cdot, t),I(\cdot, t)\right) - \left(S_{e}^*(\cdot, t), I_{e}^*(\cdot, t)\right) \right\|_{X\times X} = 0.
\]
This completes the proof.
\end{proof}

\section{Asymptotic profiles of the endemic equilibrium}
	\label{section 4}
In this section, we focus on the investigation of the asymptotic profiles of the endemic equilibrium of system (\ref{model1.2}) for small saturation parameters, and for both small and large diffusion rates. Throughout this section, we always suppose that \( d_I = d_S = d \). Under this assumption, if $R_0(d, N, m) > 1$, then it follows from Theorem \ref{theorem3.3} that system (\ref{model1.2}) admits the unique endemic equilibrium $(S_{e}^*, I_{e}^*)$. To emphasize the dependence of the endemic equilibrium  on $d$, we rewrite $(S_{e}^*, I_{e}^*)$ as $(S_{d}^*, I_{d}^*)$. According to Lemma \ref{lemma3.4}, one can obtain that $S_{d}^*(x, t) + I_{d}^*(x, t)=\frac{N}{|\Omega|}$, where $I_{d}^*$ is the unique positive $T$-periodic solution of the following equation
\begin{equation}
	\label{eq4.1}
\frac{\partial I}{\partial t} = d\mathcal{J}[I] +\left[\theta(x,t)-\gamma(x,t)-b(x,t)I\right]I, ~ x \in \Omega, t \in \mathbb{R},
\end{equation}
where $\theta(x,t)=\frac{\beta(x, t)\frac{N}{|\Omega|}}{m(x,t)+\frac{N}{|\Omega|}}$ and $b(x,t) = \frac{\beta(x, t)}{m(x,t)+\frac{N}{|\Omega|}}$.

Denote
\begin{align*}
	\Omega^+ &= \left\{ x \in \bar{\Omega} : \frac{1}{T}\int_0^T\theta(x,t)dt > \frac{1}{T}\int_0^T\gamma(x,t)dt \right\}, \\
	\Omega^- &= \left\{ x \in \bar{\Omega} : \frac{1}{T}\int_0^T\theta(x,t)dt \leq \frac{1}{T}\int_0^T\gamma(x,t)dt \right\}.
\end{align*}
For convenience, define
\[
H(x, t, \zeta) := \theta(x,t)\zeta - \gamma(x, t)\zeta - b(x,t)\zeta^2, \quad (x, t, \zeta) \in \bar{\Omega} \times \mathbb{R} \times \mathbb{R},
\]
and
\[
\tilde{H}(t, \zeta) :=\tilde{\theta}(t)\zeta - \tilde{\gamma}(t)\zeta - \tilde{b}(t)\zeta^2, \quad (t, \zeta) \in \mathbb{R} \times \mathbb{R},
\]
where
\[
\tilde{\theta}(t) = \frac{1}{|\Omega|} \int_{\Omega} \theta(x, t) dx, \quad \tilde{\gamma}(t) = \frac{1}{|\Omega|} \int_{\Omega} \gamma(x, t) dx,\quad \text{and} \quad \tilde{b}(t) = \frac{1}{|\Omega|} \int_{\Omega} b(x, t) dx.
\]

\begin{proposition}
	\label{proposition4.1}
	For each $x\in\bar{\Omega}$, the existence of the nonnegative periodic solutions for the equation
\begin{equation}
	\label{eq4.2}
	\frac{\partial \zeta}{\partial t} = H(x, t, \zeta)
\end{equation}
is given as follows.
\begin{enumerate}
	\item[{\rm (i)}]
	{\rm(\ref{eq4.2})} always has a nonnegative periodic solution 0{\rm ;}
	\item[{\rm (ii)}]
	If
	$\frac{1}{T}\int_0^T(\theta(x,t)-\gamma(x,t))dt>0$, then {\rm(\ref{eq4.2})} admits a unique positive \( T \)-periodic solution, denoted by \( \zeta^*(x, t)\){\rm ;}
	\item[{\rm (iii)}]
	If
	$\frac{1}{T}\int_0^T(\theta(x,t)-\gamma(x,t))dt\leq0$, then {\rm(\ref{eq4.2})} admits a unique nonnegative periodic solution 0.
\end{enumerate}
\end{proposition}

\begin{proof}
Statement (i) is evidently true.

For each $x\in\bar{\Omega}$, (\ref{eq4.2}) can be rewritten as
\begin{equation*}
	\frac{\partial \zeta}{\partial t} = a(x,t)\zeta-b(x,t)\zeta^{2},
\end{equation*}
where $a(x,t) = \theta(x,t)-\gamma(x,t)$. It is clear that the above equation is of Bernoulli type. Consequently, solving it yields that (\ref{eq4.2}) has a nontrivial periodic solution $\zeta(x,t)$, which is given by
\[
\zeta(x,t) = \frac{e^{\int_0^t a(x,r) dr}}{\frac{1}{\zeta_{0}(x)} + \int_0^t b(x,s)e^{\int_0^s a(x,r) dr}ds},
\]
where $\zeta_{0}(x)=\frac{e^{\int_0^Ta(x,r)dr}-1}{\int_0^T b(x,s)e^{\int_0^s a(x,r) dr}ds}\neq0$. It is easy to obtain that (\ref{eq4.2}) admits a unique positive periodic solution $\zeta^{*}(x,t)$ if and only if $\frac{1}{T}\int_0^T(\theta(x,t)-\gamma(x,t))dt>0$, where $\zeta^{*}(x,t)=\zeta(x,t)$ with $\zeta_{0}(x)>0$. Hence, if $\frac{1}{T}\int_0^T(\theta(x,t)-\gamma(x,t))dt\leq0$, (\ref{eq4.2}) admits no positive periodic solution.

The proof of (ii) and (iii) is complete.
\end{proof}

\begin{theorem}
	\label{theorem4.1}
$I_{d}^*$ is non-decreasing in $N$, that is, $I_{d,N_{1}}^*\geq I_{d,N_{2}}^*$ provided $N_{1}\geq N_{2}>0$.
\end{theorem}

\begin{proof}
Recall that $I_{d}^*(x,t)$ satisfies (\ref{eq4.1}). To highlight the dependence of $I_{d}^*(x,t)$, $\theta(x,t)$, and $b(x,t)$ on $N$, we denote them by $I_{d,N}^*$, $\theta_{N}$, and $b_{N}$. If $N_{1}\geq N_{2}>0$, then we have
\begin{equation*}
		\begin{aligned}
	\frac{\partial I_{d,N_{1}}^*}{\partial t} &= d\mathcal{J}[I_{d,N_{1}}^*] +\left[\theta_{N_{1}}-\gamma(x,t)-b_{N_{1}}I_{d,N_{1}}^*\right]I_{d,N_{1}}^*\\
	&\geq d\mathcal{J}[I_{d,N_{1}}^*] +\left[\theta_{N_{2}}-\gamma(x,t)-b_{N_{2}}I_{d,N_{1}}^*\right]I_{d,N_{1}}^*.
\end{aligned}
\end{equation*}
Going further, we obtain
\begin{equation*}
	\begin{aligned}
		&\frac{\partial I_{d,N_{1}}^*}{\partial t}-d\mathcal{J}[I_{d,N_{1}}^*] -\left[\theta_{N_{2}}-\gamma(x,t)-b_{N_{2}}I_{d,N_{1}}^*\right]I_{d,N_{1}}^*\\
		\geq&0=	\frac{\partial  I_{d,N_{2}}^*}{\partial t}-d\mathcal{J}[ I_{d,N_{2}}^*] -\left[\theta_{N_{2}}-\gamma(x,t)-b_{N_{2}} I_{d,N_{2}}^*\right] I_{d,N_{2}}^*.
	\end{aligned}
\end{equation*}
By \citep[Proposition 3.3]{Lin2023}, $I_{d,N_{1}}^*\geq I_{d,N_{2}}^*$. This completes the proof.
\end{proof}

\begin{theorem}
	\label{theorem4.1-1}
	Suppose that $\int_{\Omega}\int_0^T \beta(x,t)dtdx>\int_{\Omega}\int_0^T \gamma(x,t)dtdx$. Then the unique endemic equilibrium of system
	{\rm(\ref{model1.2})} satisfies
	\[
	\lim_{\|m\|_{C_T} \to 0} \left\| (S_d^*,I_d^*) - \left(\frac{N}{|\Omega|}-\zeta_{0},\zeta_{0}\right) \right\|_{C_T\times C_T} = 0,
	\]	
	where $\zeta_{0}(x,t)$ is the unique positive $T$-periodic solution of the following equation
	\begin{equation}
		\label{eq4.1-1}
		\frac{\partial \zeta}{\partial t} = d\mathcal{J}[\zeta] +\left[\beta(x,t)-\gamma(x,t)-\frac{|\Omega|\beta(x,t)}{N}\zeta\right]\zeta, ~ x \in \Omega, t \in \mathbb{R}.
	\end{equation}
\end{theorem}

\begin{proof} Note that $\int_{\Omega}\int_0^T \beta(x,t)dtdx>\int_{\Omega}\int_0^T \gamma(x,t)dtdx$. By virtue of \citep[Proposition 2.5]{Feng2025}, one can obtain that $\hat{R}_{0}>1$. Following the argument in \citep[Section 4]{Lin2023}, (\ref{eq4.1-1}) admits the unique positive $T$-periodic solution, denoted by $\zeta_{0}(x,t)$. Moreover,  according to Theorem \ref{theorem2.2} (iv), for sufficiently small $m\in C_T^+$, it follows that $R_{0}(d,N,m)>1$. Then, the monotonicity of $R_{0}(d,N,m)$ with respect to $m$ (see Corollary~\ref{corollary2.1}(v)) implies that for all sufficiently small $m\in C_T^+$, there exists a unique endemic equilibrium $(S_{d}^*, I_{d}^*)$ by Theorem \ref{theorem3.3}.
	
We now introduce two auxiliary problems as follows.
\begin{equation}
	\label{eq4.1-2}
	\frac{\partial \bar{\zeta}}{\partial t} = d\mathcal{J}[\bar{\zeta}] + (\beta(x,t)-\gamma(x,t)) \bar{\zeta} - b(x,t) \bar{\zeta}^2, \quad x \in \Omega, \, t > 0,
\end{equation}
and
\begin{equation}
	\label{eq4.1-3}
	\frac{\partial \underline{\zeta}}{{\rm d}t} = d\mathcal{J}[\underline{\zeta}] + (\theta(x,t)-\gamma(x,t))\underline{\zeta}  - \frac{|\Omega|\beta(x,t)}{N} \underline{\zeta}^2, \quad x \in \Omega, \, t > 0.
\end{equation}	
Recall that $\hat{R}_{0}>1$ and $R_{0}(d,N,m)>1$. From \citep[Proposition 2.4(i)]{Feng2025}, Lemma \ref{lemma2.3} and Lemma \ref{lemma3.2}, equations (\ref{eq4.1-2}) and (\ref{eq4.1-3}) admit unique positive periodic solutions, denoted by $\bar{\zeta}^*$ and $\underline{\zeta}^*$, respectively. It is easy to verify that $\underline{\zeta}^*$ and $\bar{\zeta}^*$ are a pair of sub-super solutions of (\ref{eq4.1}). According to \citep[Proposition 3.3]{Lin2023}, it follows that
\begin{equation}
	\label{th4.2-1}
	\underline{\zeta}^*(x,t) \leq I^*_{d}(x,t) \leq \bar{\zeta}^*(x,t), \quad x \in \bar{\Omega}, \, t \in [0, T].
\end{equation}
By \citep[Proposition 3.3]{Lin2023} again, we have $\bar{\zeta}^*$ is non-decreasing with respect to $m$, $\underline{\zeta}^*$ is non-increasing with respect to $m$, and $\bar{\zeta}_{m_1}^* \geq \underline{\zeta}_{m_2}^*$ for any $m_1\in  C^+_T$ and $m_2\in  C^+_T$. Therefore, for each point $(x, t) \in \bar{\Omega} \times [0, T]$,
\[
\bar{\zeta}^*(x, t) \to \zeta_{1}(x, t), \quad \text{and} \quad \underline{\zeta}^*(x, t) \to \zeta_{2}(x, t) \quad \text{as } \|m\|_{C_T} \to 0,
\]
where $\zeta_{1}(x, t)$ and $\zeta_{2}(x, t)$ satisfy (\ref{eq4.1-1}), and $\zeta_{1}(x, t)\ge \zeta_{2}(x, t)>0$. By the uniqueness of the positive $T$-periodic solution to (\ref{eq4.1-1}), we conclude that $\zeta_{0} = \zeta_{1} = \zeta_{2}$. By Dini's theorem, it follows that
\begin{equation}
	\label{th4.2-2}
	\lim_{\|m\|_{C_T} \to 0} \left\|\underline{\zeta}^* - \zeta_{0} \right\|_{C_T} = \lim_{\|m\|_{C_T} \to 0} \left\| \bar{\zeta}^* - \zeta_{0} \right\|_{C_T} = 0.
\end{equation}
From (\ref{th4.2-1}), (\ref{th4.2-2}), and $S_{d}^* + I_{d}^*=\frac{N}{|\Omega|}$, we have
	\[
\lim_{\|m\|_{C_T} \to 0} \left\| (S_d^*,I_d^*) - \left(\frac{N}{|\Omega|}-\zeta_{0},\zeta_{0}\right) \right\|_{C_T\times C_T} = 0.
\]		
\end{proof}

If $\Omega^+$ is nonempty, according to Proposition \ref{proposition4.1}, we can define the following function
\begin{equation}
	\label{eq4.3}
	\zeta_{T}(x, t) =
	\begin{cases}
		\zeta^{*}, & x\in\Omega^+, \; t \in [0, T], \\[10pt]
		0, & x\in\Omega^-, \; t \in [0, T].
	\end{cases}
\end{equation}
From the expression of $\zeta^{*}$, it is easy to verify that $\zeta_{T}\in C_{T}^{+}$ and satisfies (\ref{eq4.2}). Furthermore, it is direct from \citep[Theorem 6.2]{Shen2019} that $\zeta_{T}$ is globally asymptotically stable with respect to all initial functions which are strictly positive everywhere.

\begin{theorem}
	\label{theorem4.2}
	Assume that $\Omega^+$ is nonempty.
	The unique endemic equilibrium of system
	{\rm(\ref{model1.2})} satisfies
\[
\lim_{d \to 0} \left\| (S_d^*,I_d^*) - \left(\frac{N}{|\Omega|}-\zeta_{T},\zeta_{T}\right) \right\|_{C_T\times C_T} = 0,
\]	
where $\zeta_{T}\geq,\not\equiv0$ is given in {\rm(\ref{eq4.3})}.

\end{theorem}

\begin{proof}
Since $\Omega^+$ is nonempty, then $\max\limits_{x\in\bar{\Omega}}\frac{1}{T}\int_0^T(\theta(x,t)-\gamma(x,t))dt>0$. In view of Theorem \ref{theorem2.2}, one can obtain that
\[
\lim\limits_{d\to 0} R_0(d, N, m)
= \max\limits_{x\in\bar{\Omega}} \dfrac{\int_0^T \theta(x,t)dt}{\int_0^T \gamma(x,t)dt}>1,
\]
which indicates that there exists $\delta_{1} > 0$ such that $R_0(d, N, m)>1$ when $0 < d < \delta_{1}$. Then by Theorem \ref{theorem3.3}, there exists the unique endemic equilibrium $(S_{d}^*, I_{d}^*)$ for $0 < d < \delta_{1}$.

Consider the following two auxiliary problems
\begin{equation}
	\label{eq4.4}
	\frac{\partial \bar{I}_d}{\partial t} = d \int_{\Omega} J(x - y, t) \bar{I}_d(y, t) dy + a(x, t) \bar{I}_d - b(x,t) \bar{I}_d^2, \quad x \in \Omega, \, t > 0,
\end{equation}
and
\begin{equation}
		\label{eq4.5}
	\frac{{\rm d} \underline{I}_d}{{\rm d}t} = -d \underline{I}_d + a(x, t) \underline{I}_d - b(x,t) \underline{I}_d^2, \quad x \in \Omega, \, t > 0.
\end{equation}
Since $\Omega^+$ is nonempty, a similar analysis to that used for the existence and uniqueness of $(S_{d}^*, I_{d}^*)$ yields that there exists $\delta_{2} > 0$ such that the equation (\ref{eq4.4}) admits a unique positive T-periodic solution $\bar{I}_d^*$ when $0 < d < \delta_{2}$. Moreover, it follows from the proof of Proposition \ref{proposition4.1} that (\ref{eq4.5}) admits a nonnegative periodic solution $\underline{I}_d^*\geq,\not\equiv0$, which is defined by the same way as $\zeta_{T}$.

Let $\bar{I}_d(x,t;I_{0})$ be the solution of  (\ref{eq4.4}) with initial data $I(x,0)=I_{0}(x)$, $\underline{I}_d(x,t;I_{0})$ the solution of  (\ref{eq4.5}) with initial data $I_{0}(x)$, and $I_{d}(x,t;I_{0})$ the solution of  (\ref{eq4.1}) with initial data $I_{0}(x)$. By applying the comparison principle, we have
\begin{equation}
	\label{eq4.6}
	\underline{I}_d(x,t;I_{0}) \leq I_{d}(x,t;I_{0}) \leq \bar{I}_d(x,t;I_{0}), \quad x \in \bar{\Omega}, \, t\geq0.
\end{equation}
Recall that $R_0(d, N, m)>1$ when $0 < d < \delta_{1}$. Then we get
\begin{equation*}
	\|I_{d}(\cdot,t;I_{0}) - I_{d}^*(\cdot, t)\|_X \to 0 \quad \text{as } t \to +\infty,~\forall d \in (0, \delta_{1}).
\end{equation*}
Similarly, one can derive that
\begin{equation*}
	\|\bar{I}_d(\cdot,t;I_{0}) - \bar{I}_{d}^*(\cdot, t)\|_X \to 0 \quad \text{as } t \to +\infty,~\forall d \in (0, \delta_{2}).
\end{equation*}
Additionally, in view of \citep[Theorem 6.2]{Shen2019}, it follows that
\begin{equation*}
	\|\underline{I}_d(\cdot,t;I_{0}) - \underline{I}_{d}^*(\cdot, t)\|_X \to 0 \quad \text{as } t \to +\infty.
\end{equation*}
Let $\delta_{min} = \min\{\delta_1, \delta_2\}$. For $d \in (0, \delta_{min})$, letting $t\to +\infty$ in (\ref{eq4.6}), we obtain
\begin{equation}
	\label{eq4.7}
	\underline{I}^*_d(x,t) \leq I^*_{d}(x,t) \leq \bar{I}^*_d(x,t), \quad x \in \bar{\Omega}, \, t \in [0, T].
\end{equation}
By the comparison principle, a similar discussion as above gives that $\bar{I}_d^*$ is non-decreasing in $d$, $\underline{I}_d^*$ is non-increasing in $d$, and $\bar{I}_{d_1}^* \geq \underline{I}_{d_2}^*$ for any $d_1 > 0$ and $d_2 > 0$. Hence, for each point $(x, t) \in \bar{\Omega} \times [0, T]$,
\[
\bar{I}_d^*(x, t) \to \bar{\zeta}_{T}(x, t), \quad \text{and} \quad \underline{I}_d^*(x, t) \to \underline{\zeta}_{T}(x, t) \quad \text{as } d \to 0,
\]
where $\underline{\zeta}_{T}(x, t) \leq \bar{\zeta}_{T}(x, t)$ for each $(x, t) \in \bar{\Omega} \times [0, T]$. Clearly, $\underline{\zeta}_{T}(x, t)$ and $\bar{\zeta}_{T}(x, t)$ satisfy
\begin{equation}
	\label{eq4.8}
	\begin{cases}
	\dfrac{\partial \zeta}{\partial t} = H(x,t,\zeta)=a(x, t) \zeta - b(x,t) \zeta^2, & x \in \Omega, \; t \in \mathbb{R}, \\[6pt]
	\zeta(x, t) = \zeta(x, t + T), & x \in \Omega, \; t \in \mathbb{R}.
\end{cases}
\end{equation}
Combining Proposition \ref{proposition4.1} and the definition of $\zeta_{T}$, it is easy to verify that $\underline{\zeta}_{T} = \zeta_{T}$.  Next, we claim that $\bar{\zeta}_{T} = \zeta_{T}$. Otherwise, there exists $(x_0, t_0)$ such that $\bar{\zeta}_{T}(x_0, t_0) > \zeta_{T}(x_0, t_0) \geq 0$ since $\bar{\zeta}_{T}(x, t)\geq\underline{\zeta}_{T}(x, t)=\zeta_{T}$, which demonstrates $\frac{1}{T}\int_0^Ta(x_0,t)dt>0$. It is direct from the uniqueness
of the positive periodic solution of system (\ref{eq4.8}) and the definition of $\zeta_{T}$ that $\bar{\zeta}_{T}(x_0, t_0) = \zeta_{T}(x_0, t_0)$, a contradiction. It follows from Dini's theorem that
\begin{equation}
		\label{eq4.9}
	\lim_{d \to 0} \left\|\underline{I}_d^* - \zeta_{T} \right\|_{C_T} = \lim_{d \to 0} \left\| \bar{I}_d^* - \zeta_{T} \right\|_{C_T} = 0.
\end{equation}
By (\ref{eq4.7}) and (\ref{eq4.9}), we obtain that
\begin{equation*}
	\lim_{d \to 0} \left\| I_d^* - \zeta_{T} \right\|_{C_T} = 0.
\end{equation*}
Then, together with $S_{d}^* + I_{d}^*=\frac{N}{|\Omega|}$, we get
\[
\lim_{d \to 0} \left\| (S_d^*,I_d^*) - \left(\frac{N}{|\Omega|}-\zeta_{T},\zeta_{T}\right) \right\|_{C_T\times C_T} = 0.
\]	
\end{proof}

\begin{remark}
From Theorem {\rm\ref{theorem4.1-1}}, we see that as the saturation parameter tends to zero, the unique endemic equilibrium reduces to that of the standard-incidence model.
Based on the expression for $\zeta^*$, we find that $\zeta^*$ increases with $N$ and decreases with $m$. This implies that, in the diffusion-limited regime, increasing the total population size exacerbates disease prevalence and makes disease eradication more difficult, whereas enhancing the saturation effect suppresses disease transmission. Therefore, from the perspective of control strategies, regulating population density and strengthening the saturation effect via a series of public health interventions, such as contact restrictions and isolation policies, can both serve as effective measures to reduce the prevalence at the endemic equilibrium.
\end{remark}

\begin{theorem}
	\label{theorem4.3}
	Suppose that $\int_{\Omega}\int_0^T \theta(x,t)dtdx>\int_{\Omega}\int_0^T \gamma(x,t)dtdx$. Then the endemic equilibrium of system {\rm(\ref{model1.2})} satisfies
	\[
	\lim_{d \to +\infty} \left\| (S_d^*,I_d^*) - \left(\frac{N}{|\Omega|}-\zeta_{3},\zeta_{3}\right) \right\|_{C_T\times C_T} = 0,
	\]
	where \( \zeta_{3}(t) \) is the unique positive $T$-periodic solution of $\frac{{\rm d}\zeta}{{\rm d}t}$=$\tilde{H}(t, \zeta)$.
\end{theorem}

\begin{proof}
Since $\int_{\Omega}\int_0^T \theta(x,t)dtdx>\int_{\Omega}\int_0^T \gamma(x,t)dtdx$, Corollary \ref{corollary2.1} (ii) implies that  $R_0(d, N, m)>1$, which indicates that for each $d$, (\ref{eq3.3}) admits a unique positive solution $(S_{d}^*, I_{d}^*)$. In addition, similar to the proof of Proposition \ref{proposition4.1}, we obtain that $\frac{{\rm d}\zeta}{{\rm d}t}$=$\tilde{H}(t, \zeta)$ admits the unique positive $T$-periodic solution \( \zeta_{3}(t) \) since $\frac{1}{T}\int_0^T(\tilde{\theta}(t)-\tilde{\gamma}(t))dt>0$.

It is easy to observe that there exists a large enough positive constant $M$ independent of $d$, such that $M$ is a super-solution of (\ref{eq4.1}), that is to say,
\[
0 = \frac{\partial M}{\partial t} \geq d \mathcal{J} [M] + \frac{\beta(x, t)\frac{N}{|\Omega|}}{m(x,t)+\frac{N}{|\Omega|}} M - \gamma(x, t) M - \frac{\beta(x, t)}{m(x,t)+\frac{N}{|\Omega|}} M^2.
\]
By virtue of \citep[Proposition 3.3]{Lin2023}, we get
\begin{equation}
	\label{eq4.10}
 I_d^* (x, t) \leq M, \quad \forall (x, t) \in \bar{\Omega} \times \mathbb{R}, \, d > 0.
\end{equation}
Taking the average of equation (\ref{eq4.1}) over \(\Omega\), then
\begin{equation}
	\label{eq4.11}
\frac{1}{|\Omega|}\int_{\Omega}\frac{\partial I_d^*}{\partial t}dx = \frac{d}{|\Omega|}\int_{\Omega}\mathcal{J}[I_d^*]dx+\frac{1}{|\Omega|}\int_{\Omega}a(x,t)I_d^*dx-\frac{1}{|\Omega|}\int_{\Omega}b(x,t)(I_d^*)^{2}dx.
\end{equation}
Denote
\[
 \tilde{I}_d^*(t):=\frac{1}{|\Omega|}\int_{\Omega}I_d^*(x,t)dx,~~ \text{and}~~ \hat{I}_d^*(x,t):=I_d^*(x,t)-\tilde{I}_d^*(t).
\]
Equation (\ref{eq4.11}) can be rewritten as
\begin{equation}
	\label{eq4.12}
\frac{{\rm d} \tilde{I}_d^*}{{\rm d}t} = \tilde{H}(t, \tilde{I}_d^*) + \frac{1}{|\Omega|} \int_{\Omega} \left( H(x, t, I_d^*) - \tilde{H}(t, \tilde{I}_d^*) \right) dx.
\end{equation}
By further calculation, we have
\begin{equation}
	\label{eq4.13}
\frac{\partial \hat{I}_d^*}{\partial t} = d \mathcal{J} [\hat{I}_d^*] + H(x, t, I_d^*) - \frac{1}{|\Omega|} \int_{\Omega} H(x, t, I_d^*) dx.
\end{equation}
In view of (\ref{eq4.10}), it is easy to verify that there exist two positive constants \(C_2\) and \(C_3\) independent of \(d\), such that
\begin{equation}
	\label{eq4.14}
\left| \frac{{\rm d} \tilde{I}_d^*}{{\rm d}t}(t) \right| \leq C_2, \quad \forall t \in \mathbb{R},
\end{equation}
and
\begin{equation}
	\label{eq4.15}
\left| H(x, t, I_d^*) - \frac{1}{|\Omega|} \int_{\Omega} H(x, t, I_d^*) \, dx \right| \leq C_3, \quad \forall x \in \bar{\Omega}, \; t \in \mathbb{R}, \; d > 0.
\end{equation}
Note that $\left|\tilde{I}_d^*(t)\right|=\left|\frac{1}{|\Omega|}\int_{\Omega}I_d^*(x,t)dx\right|\leq\frac{|\Omega|}{|\Omega|}M=M$. Combined with (\ref{eq4.14}), it is direct from Ascoli-Arzela theorem that for any sequence \( d_n \to +\infty \) (as \( n \to +\infty \)), there exists a subsequence, still denoted by \( d_n \), and \( \tilde{I} \in C(\mathbb{R}) \), such that
\begin{equation}
	\label{eq4.16}
\tilde{I}_{d_n}^*(t) \to \tilde{I}(t) \text{ uniformly on } \mathbb{R} \text{ as } n \to +\infty.
\end{equation}
The remainder of the proof is divided into two steps.

\textbf{Step 1.} $\hat{I}_d^* (x, t) \to 0 \text{ uniformly on } \bar{\Omega} \times \mathbb{R} \text{ as } d \to +\infty.$

Set \( l(t) := \int_{\Omega} \left( \hat{I}_d^*(x, t) \right)^2 dx \). Clearly, $l(t)$ is $T$-periodic. Direct calculation implies that
\begin{align*}
	\frac{dl(t)}{dt}
	&= 2 \int_{\Omega} \hat{I}_d^* (x, t) \frac{\partial \hat{I}_d^* (x, t)}{\partial t} dx \\
	&= 2 \int_{\Omega} \hat{I}_d^* (x, t) \left(d \mathcal{J} [\hat{I}_d^*] + H(x, t, I_d^*) - \frac{1}{|\Omega|} \int_{\Omega} H(x, t, I_d^*) dx \right) dx \\
	&\leq 2d \int_{\Omega} \hat{I}_d^* (x, t) \left[ \int_{\Omega} J(x - y, t) (\hat{I}_d^* (y, t) - \hat{I}_d^* (x, t)) dy \right] dx + 2C_3 \int_{\Omega} \left| \hat{I}_d^* (x, t) \right| dx \\
	&= -d \int_{\Omega} \int_{\Omega} J(x - y, t) \left[ \hat{I}_d^* (y, t) - \hat{I}_d^* (x, t) \right]^2 dy dx + 2C_3 \int_{\Omega} \left| \hat{I}_d^* (x, t) \right| dx \\
	&\leq -2d \alpha_0 l(t) + C_4,
\end{align*}
for some positive constant \(C_4\) independent of \(d, t\) and \(x\), where \(\alpha_0\) is given in Lemma \ref{lemma2.1}. According to the constant-variation formula and the comparison principle, we have
\begin{equation}
	\label{eq4.17}
	l(t) \leq e^{-2d \alpha_0 t} l(0) + \int_0^t e^{-2d \alpha_0 (t-s)} C_4 ds, \quad \forall t \geq 0.
\end{equation}
If $t=0$, then $\lim_{d \to +\infty} \int_0^t e^{-2d \alpha_0 (t-s)} C_4 ds=0$. If $t>0$, then
\[
\lim_{d \to +\infty} \int_0^t e^{-2d \alpha_0 (t-s)} C_4 ds=\lim_{d \to +\infty}\frac{C_4-C_4e^{-2d \alpha_0 t}}{2d \alpha_0}=\lim_{d \to +\infty}\frac{2 \alpha_0 tC_4e^{-2d \alpha_0 t}}{2\alpha_0}=0.
\]
Thus, we have
\begin{equation}
	\label{eq4.18}
 l(t)  \to 0 ~\text{ as } ~d \to +\infty.
\end{equation}
Furthermore, it follows that for any $\varepsilon>0$, there exists a $d_{0}$ such that
\begin{align*}
	l(t) \leq \varepsilon^2, \quad \forall d \geq d_0, \; t \in \mathbb{R}.
\end{align*}
By (\ref{eq4.13}) and the constant-variation formula, we derive that
\begin{align*}
	\left| \hat{I}_d^*(x, t) \right|
	=&	\left|e^{-d\eta(x,t,0)}\hat{I}_d^*(x, 0)+\int_0^te^{-d\eta(x,t,s)}\left[d\int_{\Omega}J(x - y, s)\hat{I}_d^*(y, s)dy\right.\right.  \\
	&\left.\left.+H(x, s, I_d^*) - \frac{1}{|\Omega|} \int_{\Omega} H(x, s, I_d^*) dx\right]ds\right|\\
	\leq& e^{-d\eta(x, t, 0)} \left| \hat{I}_d^*(x, 0) \right|
	+ d \int_0^t e^{-d\eta(x, t, s)} \int_{\Omega} J(x - y, s) \left| \hat{I}_d^*(y, s) \right| dy \, ds \\
	& + \int_0^t e^{-d\eta(x, t, s)} \left| H(x, s, I_d^*(x, s)) - \frac{1}{|\Omega|} \int_{\Omega} H(x, s, I_d^*(x, s)) dx \right| ds \\
	\leq& e^{-d \alpha_0 t} \left| \hat{I}_d^*(x, 0) \right|
	+ d \int_0^t e^{-d \alpha_0 (t-s)} \int_{\Omega} J(x - y, s) \left| \hat{I}_d^*(y, s) \right| dy \, ds+C_3 \int_0^t e^{-d \alpha_0 (t-s)}\, ds\\
	\leq& e^{-d \alpha_0 t} \left| \hat{I}_d^*(x, 0) \right|
	+ C_5 d \int_0^t e^{-d \alpha_0 (t-s)} l^{\frac{1}{2}}(s) \, ds
	+ C_3 \int_0^t e^{-d \alpha_0 (t-s)} \, ds,
\end{align*}
for some positive constant \(C_5\) independent of \(d, t\) and \(x\), where $\eta(x, t, s) = \int_s^t \int_{\Omega} J(x - y, r) \, dy \, dr$, and \(\alpha_0\) is given in Lemma \ref{lemma2.1} satisfying
\begin{equation*}
	\int_{\Omega} J(x - y, t) \, dy \geq \alpha_0, \quad \forall x \in \bar{\Omega}, \; t \in \mathbb{R}.
\end{equation*}
Recall that $l(t) \leq \varepsilon^2$, then a simple calculation gives
\[
d \int_{0}^{t} e^{-d \alpha_0 (t-s)} l^{\frac{1}{2}} (s) \, ds \leq d \int_{0}^{t} e^{-d \alpha_0 (t-s)} \varepsilon \, ds \leq \frac{\varepsilon(1-e^{-d \alpha_0t})}{\alpha_0}\leq\frac{\varepsilon}{\alpha_0}, \quad \forall t \geq 0,~d\geq d_{0}.
\]
By the arbitrariness of $\varepsilon$, it follows that
\[
d \int_{0}^{t} e^{-d \alpha_0 (t-s)} l^{\frac{1}{2}}(s) \, ds \to 0 \text{ uniformly on } [0, +\infty) \text{ as } d \to +\infty.
\]
Similarly, one can get that
\[
e^{-d \alpha_0 t} \left| \hat{I}_d^* (x, 0) \right| \to 0 \text{ uniformly on } \bar{\Omega} \times [T, +\infty) \text{ as } d \to +\infty,
\]
and
\[
\int_0^t e^{-d \alpha_0 (t-s)} \, ds \to 0 \text{ uniformly on } [0, +\infty) \text{ as } d \to +\infty.
\]
Hence,
\[
\hat{I}_d^* (x, t) \to 0 \text{ uniformly on } \bar{\Omega} \times [T, +\infty) \text{ as } d \to +\infty.
\]
Combining with the \(T\)-periodicity of \(\hat{I}_d^* (x, t)\), we obtain
\[
\hat{I}_d^* (x, t) \to 0 \text{ uniformly on } \bar{\Omega} \times \mathbb{R} \text{ as } d \to +\infty.
\]

\textbf{Step 2.} $\tilde{I}_d^* (t) \to \zeta_{3}(t) \text{ uniformly on } \mathbb{R} \text{ as } d \to +\infty.$

Note that
\begin{align*}
	&\left| \frac{1}{|\Omega|} \int_{\Omega} \left( H(x, t, I_d^*) - \tilde{H}(t, \tilde{I}_d^*) \right) dx \right|\\
	=& \left| \frac{1}{|\Omega|} \int_{\Omega} \left\{ (\theta(x, t) - \gamma(x, t)) I_d^*(x, t)-\left(\tilde{\theta}(t) - \tilde{\gamma}(t)\right)\tilde{I}_d^*(t)- b(x,t)(I_d^*(x, t))^2+\tilde{b}(t)\left( \tilde{I}_d^*(t) \right)^2  \right\} dx \right| \\
	\leq& C_6 l^{\frac{1}{2}}(t),
\end{align*}
for some positive constant \( C_6 \) independent of \( d, x \) and \( t \). Together with (\ref{eq4.18}), we have
\begin{equation}
	\label{eq4.19}
	\int_{0}^{T} \left| \frac{1}{|\Omega|} \int_{\Omega} \left( H(x, t, I_d^*) - \tilde{H}(t, \tilde{I}_d^*) \right) dx \right| dt \to 0~ \text{ as }~ d \to +\infty.
\end{equation}
Let \( d_n \) be any sequence with \( d_n \to +\infty \) as \( n \to +\infty \). Integrating (\ref{eq4.12}) from 0 to \( t \) and replacing \( d \) by \( d_n \), it follows that
\[
\tilde{I}_{d_{n}}^*(t) - \tilde{I}_{d_{n}}^*(0) = \int_0^t \tilde{H}(s, \tilde{I}_{d_{n}}^*(s)) \, ds+\int_0^t\frac{1}{|\Omega|} \int_{\Omega} \left( H(x, t, \tilde{I}_{d_{n}}^*) - \tilde{H}(t, \tilde{I}_{d_{n}}^*) \right) dxdt, \quad \forall t \in [0, T].
\]
Let $n\to+\infty$, it is direct from (\ref{eq4.16}) and (\ref{eq4.19}) that
\[
\tilde{I}(t) - \tilde{I}(0) = \int_0^t \tilde{H}(s, \tilde{I}(s)) \, ds, \quad \forall t \in [0, T].
\]
Consequently, $\tilde{I}(t)$ satisfies $\frac{{\rm d}\zeta}{{\rm d}t}$=$\tilde{H}(t, \zeta)$. Since $I_d^*(x,t)>0$. we have $\tilde{I}(t)\geq0$. Recall that $\frac{1}{T}\int_0^T(\tilde{\theta}(t)-\tilde{\gamma}(t))dt>0$, similar to the analysis in Proposition \ref{proposition4.1}, we obtain that the only possible cases are
\[
\tilde{I}(t) \equiv 0 \quad \text{or} \quad \tilde{I}(t) \equiv \zeta_3(t).
\]

We now argue $\tilde{I}(t) \not\equiv 0$ by contradiction. Assume that $\tilde{I}(t) \equiv 0$. Then, by the proof of Step 1 and (\ref{eq4.16}), one can obtain that
\[
I_{d_{n}}^* (x, t) \to 0 \text{ uniformly on } \bar{\Omega} \times \mathbb{R} \text{ as } n \to +\infty.
\]
Define $M_{n}:= \max_{(x,t) \in \bar{\Omega}\times [0,T]} I_{d_{n}}^* (x, t)$. Then we have $M\geq M_{n}>0$ and $M_{n}\to0$ as $n \to +\infty$. Let
\[
I_{n}(x,t):=\frac{I_{d_{n}}^* (x, t)}{M_{n}}.
\]
Then $0<I_{n}\leq1$. In view of (\ref{eq4.1}), $I_{n}(x,t)$ satisfies
\begin{equation}
	\label{eq4.20}
	\frac{\partial I_{n}}{\partial t} = d_{n}\mathcal{J}[I_{n}] +(\theta(x,t)-\gamma(x,t))I_{n}-M_{n}b(x,t)I_{n}^{2}.
\end{equation}
By a similar argument as in Step 1,  we get
\[
I_n (x, t) -\frac{1}{|\Omega|}  \int_{\Omega} I_n (x, t)dx \to 0 ~\text{ uniformly on } \bar \Omega\times \mathbb{R}  \text{ as }~ n \to +\infty.
\]
Moreover, from the derivation of (\ref{eq4.16}), there exists a subsequence, still denoted by $d_{n}$, and $\hat{I}(t)\in C(\mathbb{R})$ such that
\begin{equation*}
	\tilde{I}_n (t):=\frac{1}{|\Omega|}  \int_{\Omega} I_n (x, t)dx \to \hat{I}(t) \text{ uniformly on } \mathbb{R} \text{ as } n \to +\infty.
\end{equation*}
Let $\hat I$ attain its maximum at $t=t_{*}$, i.e., $\hat{I}(t_{*})=1$. Intergrating both sides of (\ref{eq4.20}) over $\Omega$ and letting $n \to +\infty$, it follows that
\begin{equation*}
	\frac{\partial \hat{I}(t)}{\partial t} =(\tilde{\theta}(t)-\tilde{\gamma}(t)) \hat{I}(t).
\end{equation*}
Solving the above equation gives
\begin{equation}
	\label{eq4.21}
\hat{I}(t) = \hat{I}(t_*) e^{ \int_{t_*}^t \left( \tilde{\theta}(s) - \tilde{\gamma}(s) \right) ds }, \quad t \geq t_*.
\end{equation}
Since $I_{n}(x,t)$ is $T$-periodic, $\hat{I}(t)$ is also $T$-periodic. Therefore, $\hat{I}(t_*)=\hat{I}(t_*+T)$. Substituting $t=t_*+T$ into (\ref{eq4.21}) yields
\[
 \int_0^T \left( \tilde{\theta}(s) - \tilde{\gamma}(s) \right) ds =0,
\]
which contradicts the assumption that $\int_{\Omega}\int_0^T \theta(x,t)dtdx>\int_{\Omega}\int_0^T \gamma(x,t)dtdx$. Hence, $\tilde{I}(t) \equiv \zeta_3(t)$. By (\ref{eq4.16}) and the arbitrariness of $d_{n}$, we derive
\[
\tilde{I}_d^* (t) \to \zeta_{3}(t) \text{ uniformly on } \mathbb{R} \text{ as } d \to +\infty.
\]
Combined with Step 1, it follows that
\[
\lim_{d \to +\infty} \left\| I_d^* - \zeta_{3}\right\|_{C_T} = 0.
\]
Then
\[
\lim_{d \to +\infty} \left\| S_d^* - \left(\frac{N}{|\Omega|}-\zeta_{3}\right) \right\|_{C_T} = 0.
\]
This completes the proof.
\end{proof}

\begin{remark}
By an argument similar to that in Theorem 	{\rm\ref{theorem4.3}}, the stronger condition $\beta(x,t)>\gamma(x,t)$ for all $(x,t)\in\overline{\Omega}\times\mathbb{R}$ in {\rm\citep[Theorem 4.2]{Lin2023}} can be relaxed to the weaker condition that $\int_{\Omega}\int_0^T \beta(x,t)\,dt\,dx>\int_{\Omega}\int_0^T \gamma(x,t)\,dt\,dx$.
\end{remark}

\section{Numerical simulations }
	\label{section 5}
In this section, some numerical simulations are presented to illustrate the theoretical results. First, we give multiple sets of parameter values and initial data in Tables \ref{tab:common_params}-\ref{tab:initial_conditions}.

\begin{table}[H]
	\centering
	\caption{Common parameter values used in all simulations}
	\begin{tabular}{c|c|c|c}
		\toprule
		$\Omega$ & $T$ & $J(x-y,t)$ & $\sigma(t)$ \\
		\midrule
		$(-1, 1)$ & $12$ & $\displaystyle \frac{1}{\sqrt{\pi }\sigma(t)} \exp\left(-\frac{(x-y)^2}{\sigma^2(t)}\right)$ & $\displaystyle 1 + 0.08 \sin\left(\frac{2\pi t}{12}\right)$ \\
		\bottomrule
	\end{tabular}
	\label{tab:common_params}
\end{table}

\begin{table}[H]
	\centering
	\captionsetup{labelsep=period, font=small}
	\caption{Six sets of parameter values for $\beta$, $\gamma$, and $m$}
	\setlength{\tabcolsep}{2pt}
	\begin{tabular}{c|c|c|c}
		\toprule
		Set & $\beta(x,t)$ & $\gamma(x,t)$ & $m(x,t)$ \\
		\midrule
		1 &
		$10 + \sin(\pi x) + 0.2 \sin(2\pi t/12)$ &
		$2 + \sin(\pi x) + 0.15 \cos(2\pi t/12)$ &
		$6 + \cos(\pi x) + 0.25 \sin(2\pi t/12)$ \\
		2 &
		$8 + \sin(\pi x) + 0.2 \sin(2\pi t/12)$ &
		$3.3 + \sin(\pi x) + 0.15 \cos(2\pi t/12)$ &
		$6 + \cos(\pi x) + 0.25 \sin(2\pi t/12)$ \\
		3 &
		$12 + \sin(\pi x) + 0.2 \sin(2\pi t/12)$ &
		$9 + \sin(\pi x) + 0.15 \cos(2\pi t/12)$ &
		$3 + \cos(\pi x) + 0.25 \sin(2\pi t/12)$ \\
		 4 &
		$3 + \sin(\pi x) + 0.2 \sin(2\pi t/12)$ &
		$7 + \sin(\pi x) + 0.15 \cos(2\pi t/12)$ &
		$6 + \cos(\pi x) + 0.25 \sin(2\pi t/12)$ \\
		 5 &
		$2 + \sin(\pi x) + 0.15 \cos(2\pi t/12)$ &
		$2 + \sin(\pi x) + 0.15 \cos(2\pi t/12)$ &
		$6 + \cos(\pi x) + 0.25 \sin(2\pi t/12)$ \\
		 6 &
		$\begin{array}{@{}c@{}}
			\frac{1}{4}\big[2 + \sin(\pi x) + 0.15 \cos(2\pi t/12)\big] \\
			\times \big[10 + \cos(\pi x) + 0.25 \sin(2\pi t/12)\big]
		\end{array}$ &
		$2 + \sin(\pi x) + 0.15 \cos(2\pi t/12)$ &
		$6 + \cos(\pi x) + 0.25 \sin(2\pi t/12)$ \\
		\bottomrule
	\end{tabular}
	\label{tab:params_six_sets}
\end{table}

\begin{table}[H]
	\centering
	\captionsetup{labelsep=period, font=small}
	\caption{Four sets of initial conditions used in numerical simulations}
	\setlength{\tabcolsep}{4pt}
	\begin{tabular}{c|c|c}
		\toprule
		Initial data & $S_0(x)$ & $I_0(x)$ \\
		\midrule
		1 & $3.9 + \cos(\pi x)$ & $0.1 + 0.1 \cos(\pi x)$ \\
		2 & $2.8 + 0.65 \sin(\pi x)$ & $1.2 + 0.14 \sin(\pi x)$ \\
		3 & $2.0 + 0.7 \cos(\pi x)$ & $2.0 + 0.12 \sin(\pi x)$ \\
		4 & $0.6 + 0.18 \sin(\pi x)$ & $3.4 + 0.3 \cos(\pi x)$ \\
		\bottomrule
	\end{tabular}
	\label{tab:initial_conditions}
\end{table}


To investigate the effect of the diffusion coefficient for infectious individuals, $d_I$, on the basic reproduction number $R_0$, we treat $d_I$ as a variable and fix the total population size $N=8$. We adopt the common parameter values listed in Table \ref{tab:common_params}, while $\beta$, $\gamma$ and $m$ take the values from parameter sets 1, 2 and 3 in Table \ref{tab:params_six_sets}, respectively, to generate Fig. \ref{Fig1}. These three parameter sets correspond to $R_0>1$, $R_0$ crossing 1, and $R_0<1$, which are illustrated in subplots (a), (b) and (c) of Fig. \ref{Fig1}, respectively. In addition, according to Theorem \ref{theorem2.2} (ii), the limiting values of $R_0$ under each parameter set are calculated and labelled in the figure. From Fig. \ref{Fig1}, we observe that $R_0$ decreases as $d_I$ increases for all three groups of parameters. However, since we plot the results using discrete sample points with limited numerical accuracy, we cannot rigorously claim that $R_0$ is monotone in $d_I$ in the absence of a theoretical verification.

\begin{figure}[htpp]
	\centering
	\subfigure[]
	{
		\begin{minipage}{8.16cm}
			\centering
			\includegraphics[width=1\linewidth]{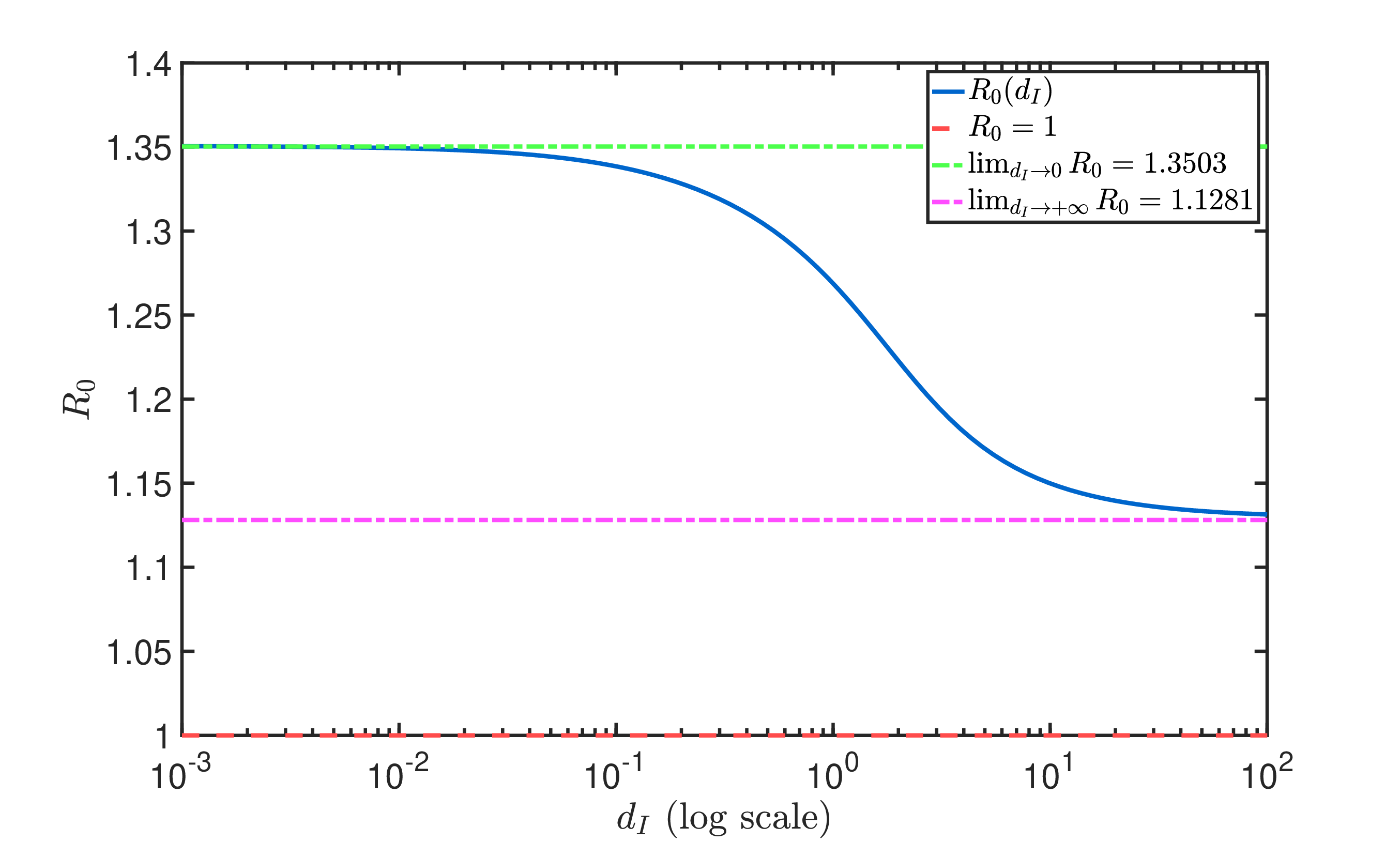}
		\end{minipage}
	}
	\subfigure[]
	{
		\begin{minipage}{8.16cm}
			\centering
			\includegraphics[width=1\linewidth]{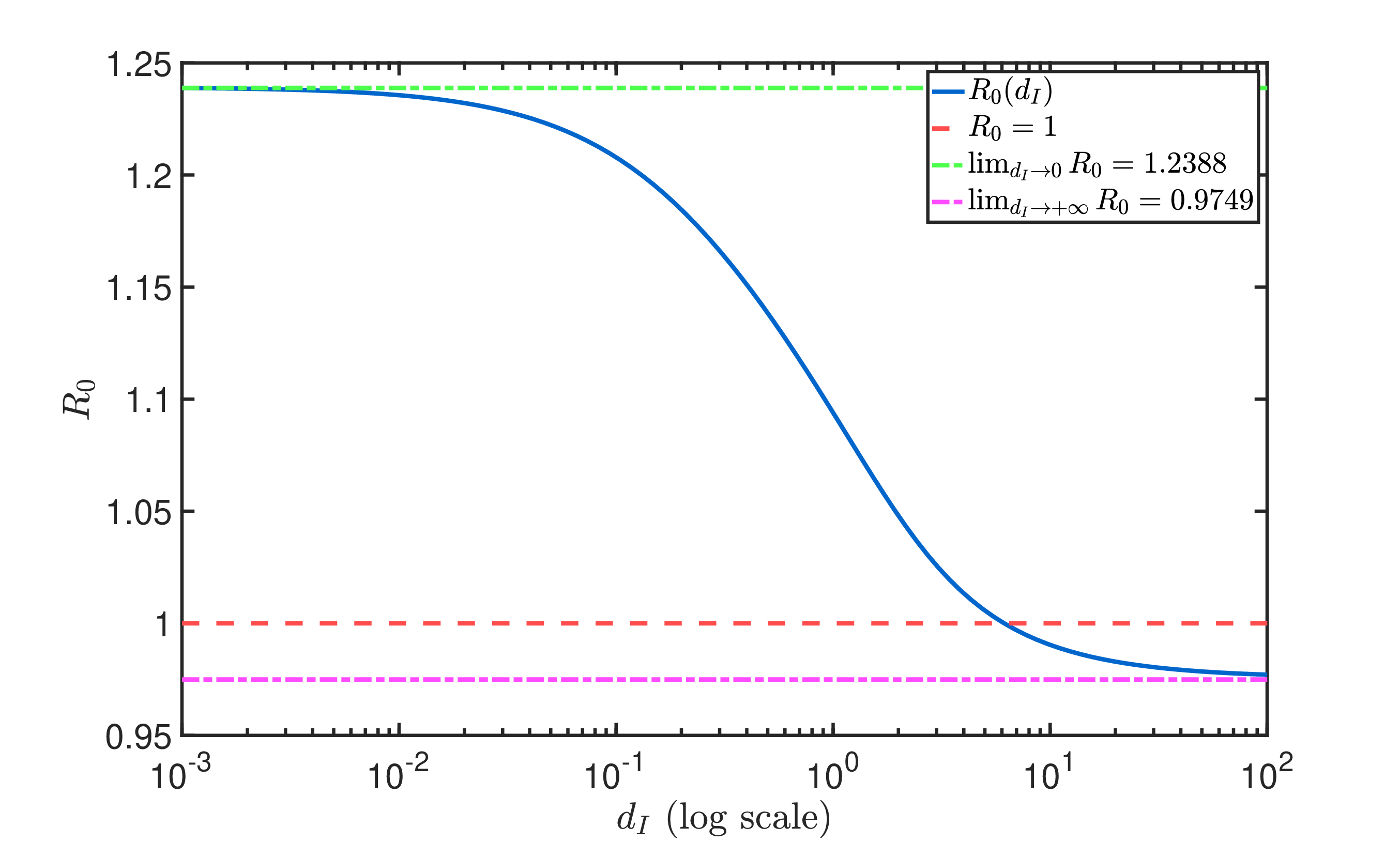}
		\end{minipage}
	}
	\subfigure[]
{
	\begin{minipage}{8.16cm}
		\centering
		\includegraphics[width=1\linewidth]{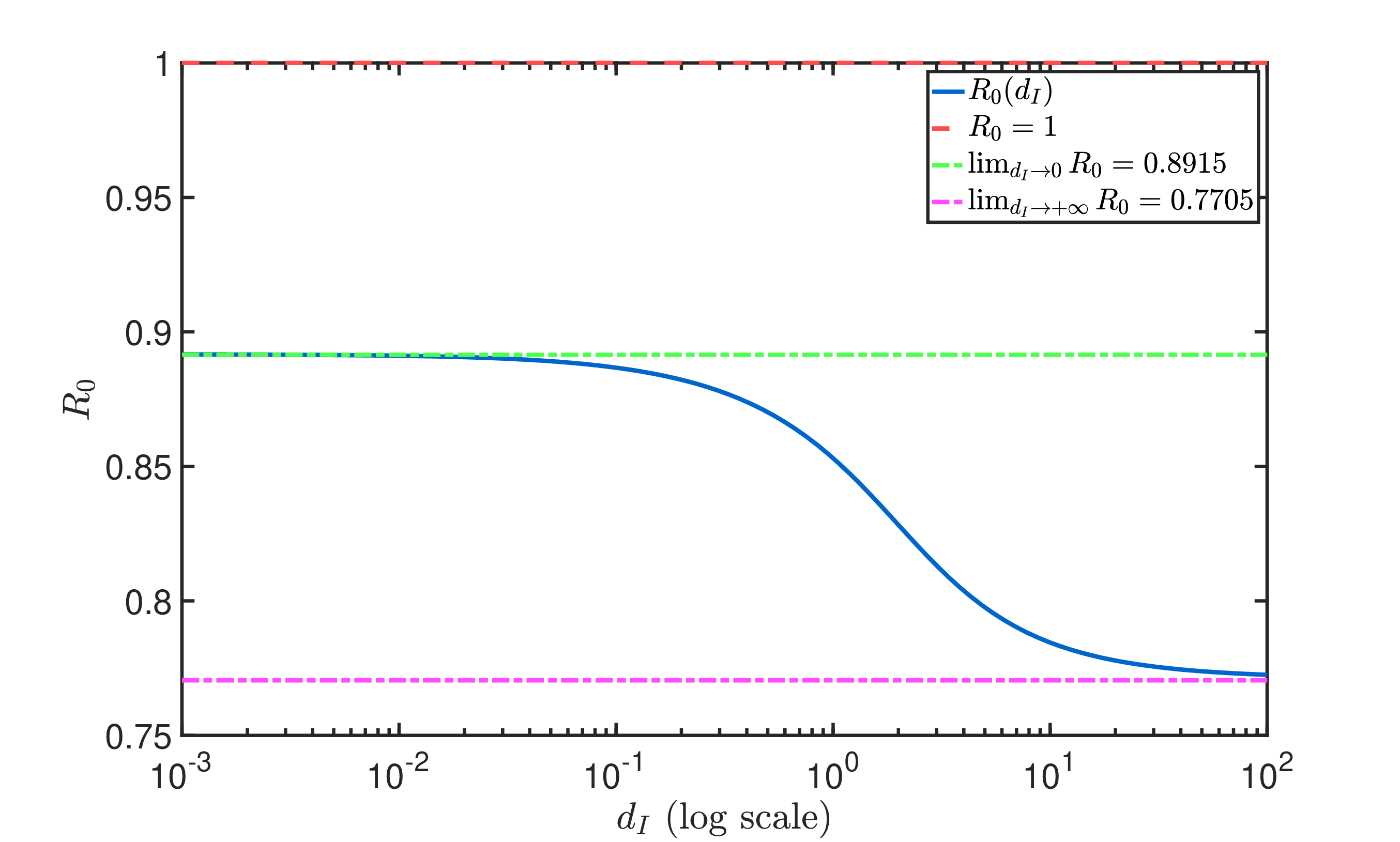}
	\end{minipage}
}
	\renewcommand{\figurename}{\footnotesize{\textbf{Fig.}}}
	\caption{\footnotesize{The dependence of $R_{0}$ on $d_{I}$ under different parameter sets. (a) $R_{0}>1$, (b) $R_{0}$ crosses the threshold 1, (c) $R_{0}<1$.}}
	\label{Fig1}
\end{figure}

Next, we focus on exploring the influence of the total population size $N$ on the basic reproduction number $R_0$. We set $d_I=1$, adopt the common parameters in Table \ref{tab:common_params}, and take $\beta$, $\gamma$, $m$ from parameter set 1 of Table \ref{tab:params_six_sets}. Treating $N$ as a free parameter and holding other parameters fixed, we plot Fig. \ref{Fig2} (a). We observe that $R_0$ grows with $N$, consistent with Corollary \ref{corollary2.1} (iv) in the theoretical section. Biologically, a larger total population size facilitates disease outbreaks. Therefore, limiting the total population size is beneficial for disease control. Furthermore, for this parameter set, a direct computation gives $\lim_{N\to 0} R_0(d_{I}, N, m)
= 0$, $\lim_{N\to +\infty} R_0(d_{I}, N, m)
= \hat{R}_0=6.5083$ in accordance with Theorem \ref{theorem2.2} (iii). As shown in Fig. \ref{Fig2} (a), the $R_0$-versus-$N$ curve exhibits two distinct asymptotic regimes, approaching different limiting values for small and large $N$, respectively.

According to Corollary \ref{corollary2.1} (v), $R_0$ is monotone decreasing with respect to $m(x,t)$. To visualize this monotonicity, we introduce a constant scaling coefficient $m_0$ and define $M_{m_0} = m_0 \cdot m(x,t)$. Treating $m_0$ as a variable, fixing $N = 8$, and keeping all other parameters the same as those used in Fig.~\ref{Fig2} (a), we replace $m$ with $M_{m_0}$ in the definition of $R_0$ and obtain Fig.~\ref{Fig2} (b). The figure shows that $R_0$ decreases as $m_0$ increases. From a biological perspective, this implies that enhancing the saturation effect through a series of public health interventions, such as contact restrictions and isolation policies, facilitates disease elimination. Furthermore, the $R_0$-versus-$m_0$ curve approaches the limiting values $\hat{R}_0=6.5083$ and $0$ at the two ends, respectively, which are computed according to Theorem~\ref{theorem2.2} (iv).

\begin{figure}[htpp]
	\centering
	\subfigure[]
	{
		\begin{minipage}{8cm}
			\centering
			\includegraphics[width=1\linewidth]{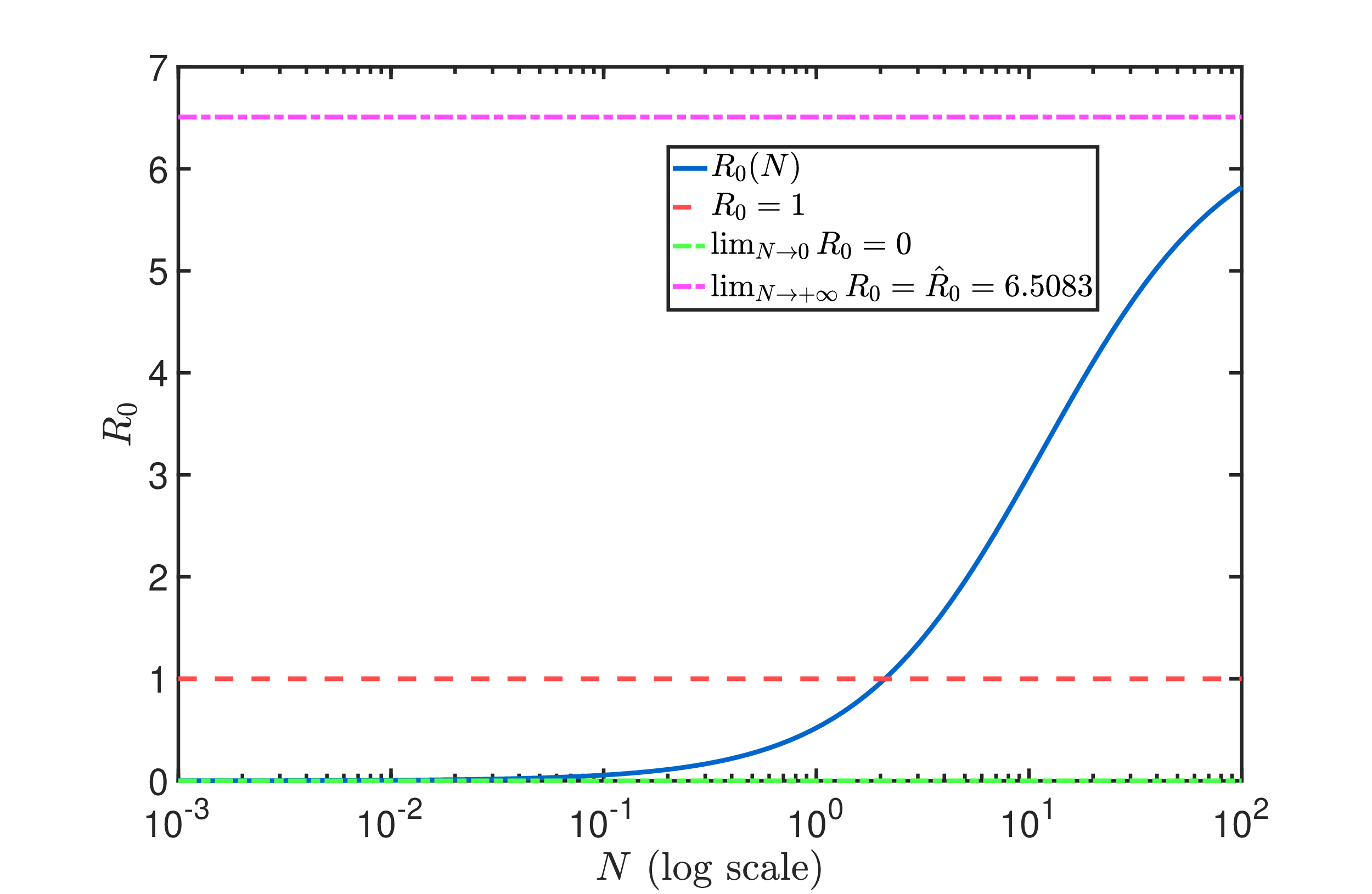}
		\end{minipage}
	}
	\subfigure[]
	{
		\begin{minipage}{8cm}
			\centering
			\includegraphics[width=1\linewidth]{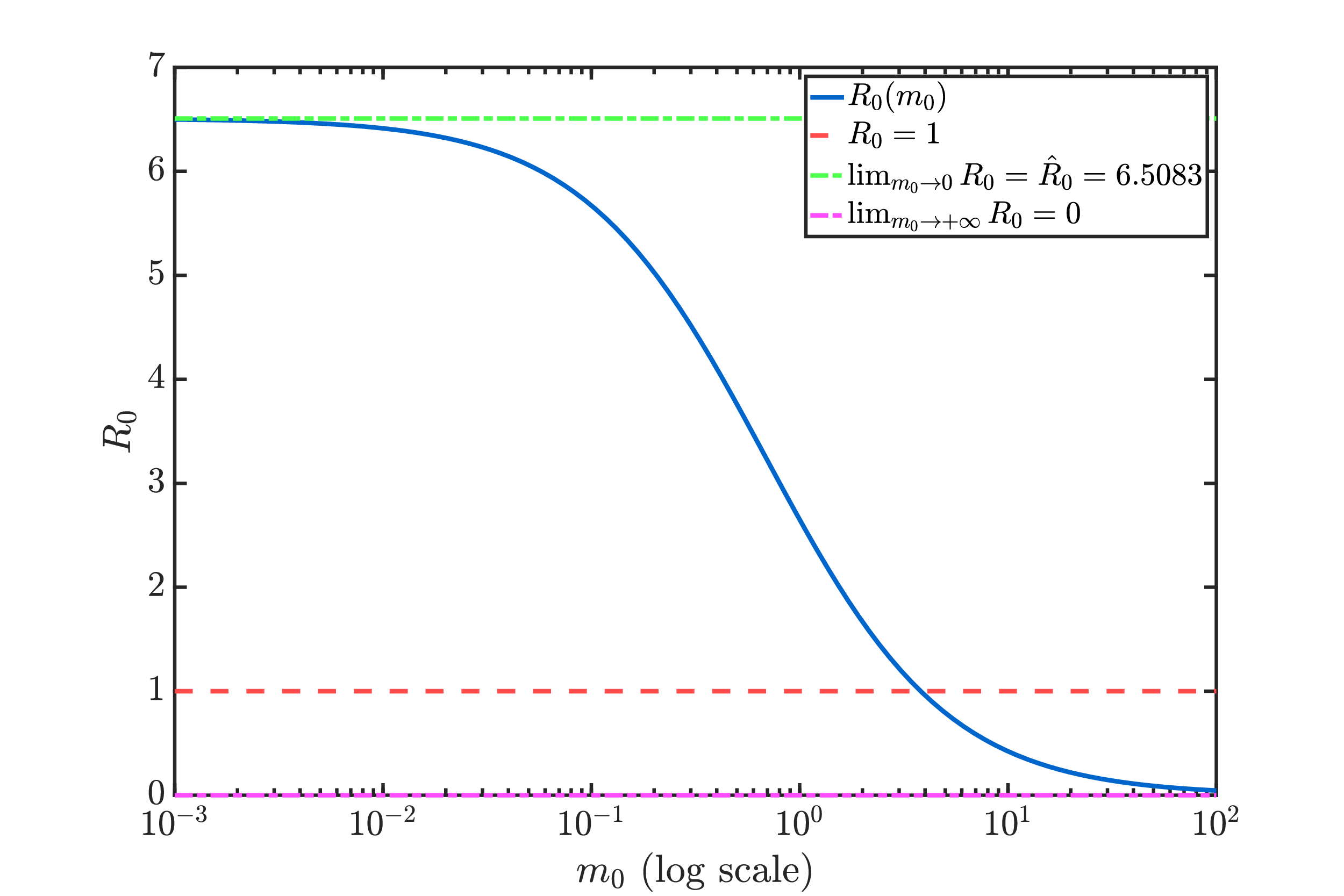}
		\end{minipage}
	}
	\renewcommand{\figurename}{\footnotesize{\textbf{Fig.}}}
	\caption{\footnotesize{The dependence of $R_{0}$ on $N$ and $m_{0}$. (a) $R_{0}(N)$, (b) $R_{0}(m_{0})$.}}
	\label{Fig2}
\end{figure}

For system (\ref{model1.2}), we set $d_{S}=0.8$, $d_{I}=1$, $N=8$, use the common parameters listed in Table~\ref{tab:common_params}, and choose parameter sets~4 and 5 in Table~\ref{tab:params_six_sets} for $\beta$, $\gamma$, and $m$, which correspond to the cases $\beta < \gamma$ and $\beta = \gamma$ for all $(x, t) \in \bar{\Omega} \times [0, T]$, respectively. According to Proposition~\ref{proposition3.2} and Theorem~\ref{theorem3.1}, there exists a unique disease-free equilibrium $(4, 0)$, and it is globally attractive, as illustrated in Figs.~\ref{Fig3} and \ref{Fig4}. It demonstrates that for three distinct spatially heterogeneous initial conditions (initial data 1, 2, and 4 in Table~\ref{tab:initial_conditions}), the solutions of system (\ref{model1.2}) all converge to the disease-free equilibrium $(4, 0)$ provided $\beta \leq \gamma$ for all $(x, t) \in \bar{\Omega} \times [0, T]$. The theorem indicates that if the transmission rate is no greater than the recovery rate at every point in space and time, the disease cannot persist in the population. Consequently, regardless of the initial infection level, the infected population eventually vanishes as time tends to infinity.

	\begin{figure}[htpp]
	\centering
	\begin{minipage}[t]{1\linewidth}
		\centering
		\includegraphics[height=0.32\linewidth,width=0.99\linewidth]{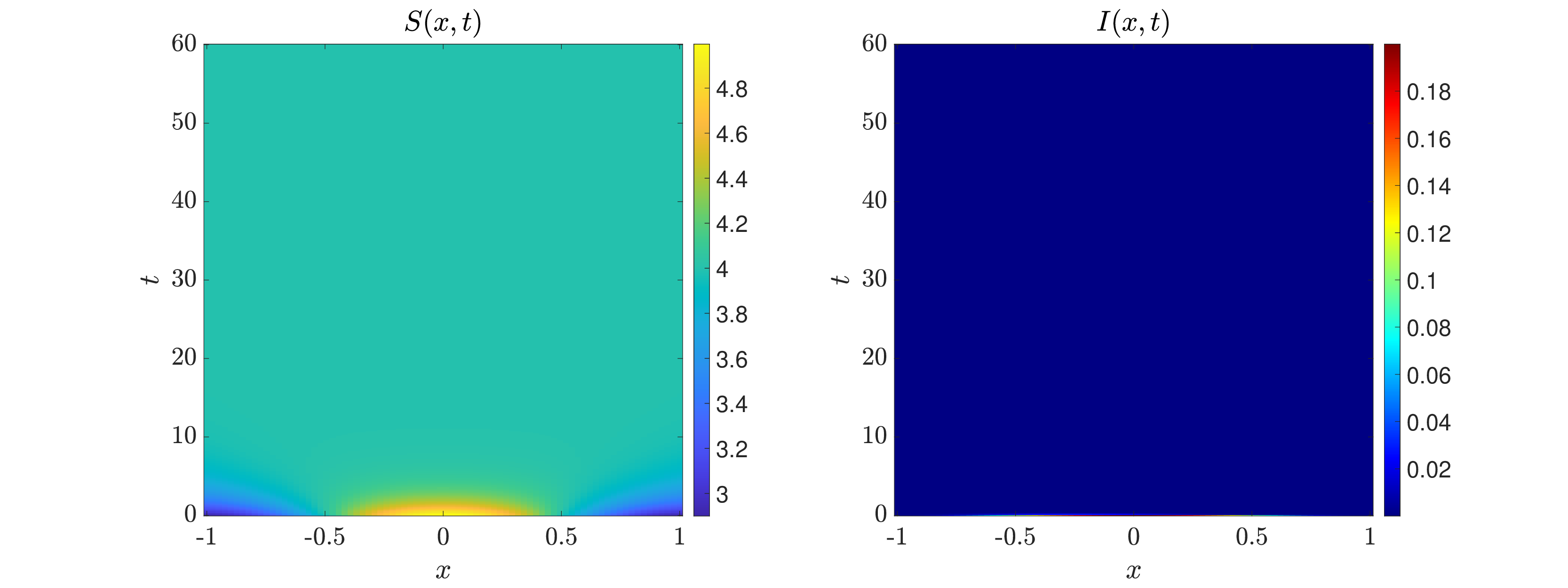}
		\centerline{\scriptsize{(a) $S_{0}(x)=3.9+\cos(\pi x)$, $I_{0}(x)=0.1+0.1\cos(\pi x)$} }
	\end{minipage}
	\begin{minipage}[t]{1\linewidth}
		\centering
		\includegraphics[height=0.32\linewidth,width=0.99\linewidth]{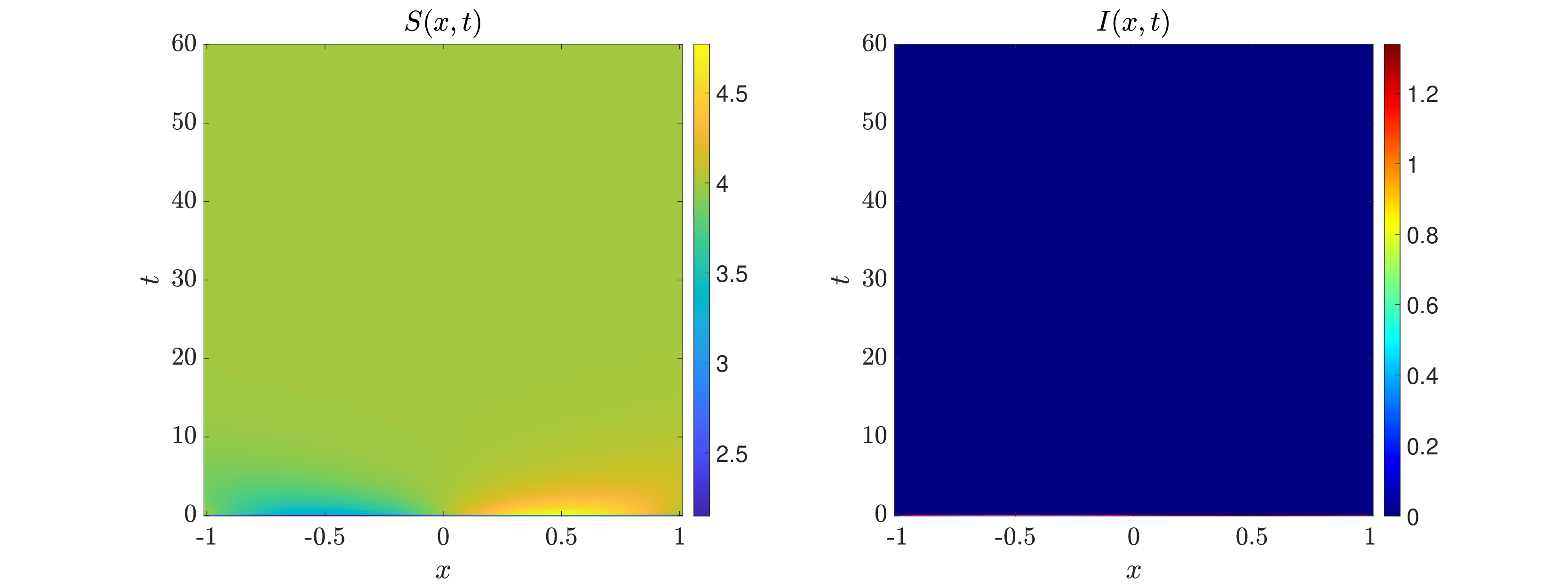}
		\centerline{\scriptsize{(b) $S_{0}(x)=2.8+0.65\sin(\pi x)$, $I_{0}(x)=1.2+0.14\sin(\pi x)$} }
	\end{minipage}
	\begin{minipage}[t]{1\linewidth}
	\centering
	\includegraphics[height=0.32\linewidth,width=0.99\linewidth]{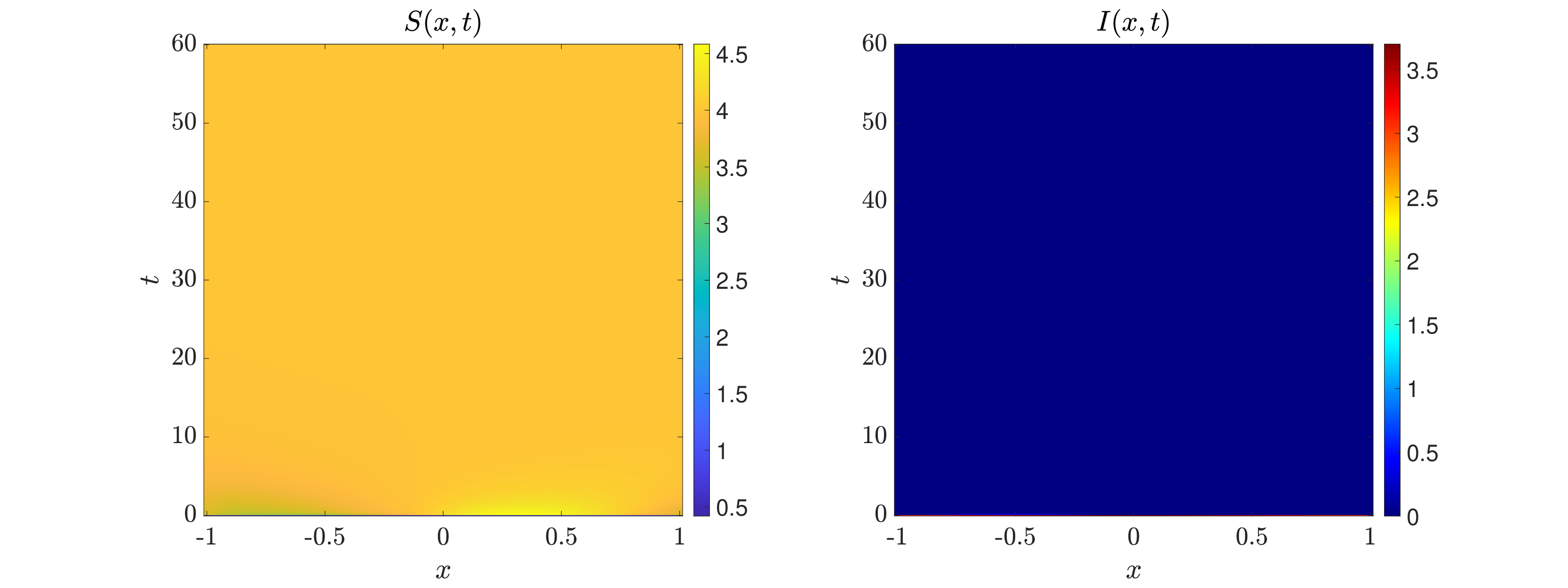}
	\centerline{\scriptsize{(c) $S_{0}(x)=0.6+0.18\sin(\pi x)$, $I_{0}(x)=3.4+0.3\cos(\pi x)$} }
\end{minipage}
	\renewcommand{\figurename}{\footnotesize{\textbf{Fig.}}}
	\caption{\footnotesize{Spatiotemporal profiles of system (\ref{model1.2}) with $\beta(x,t)<\gamma(x,t)$ under three distinct spatially heterogeneous initial conditions. (a) Initial data 1, (b) Initial data 2, (c) Initial data 4.}}
	\label{Fig3}
\end{figure}

\begin{figure}[htpp]
	\centering
	\begin{minipage}[t]{1\linewidth}
		\centering
		\includegraphics[height=0.32\linewidth,width=0.99\linewidth]{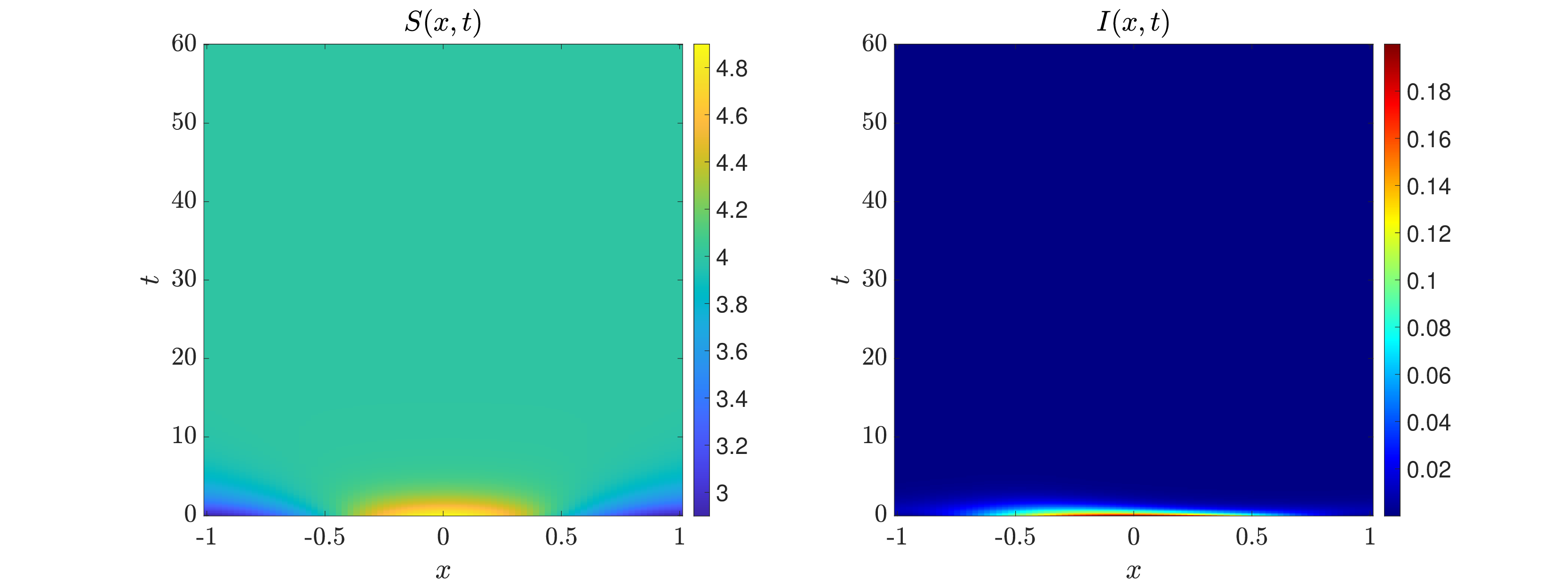}
		\centerline{\scriptsize{(a) $S_{0}(x)=3.9+\cos(\pi x)$, $I_{0}(x)=0.1+0.1\cos(\pi x)$} }
	\end{minipage}
	\begin{minipage}[t]{1\linewidth}
		\centering
		\includegraphics[height=0.32\linewidth,width=0.99\linewidth]{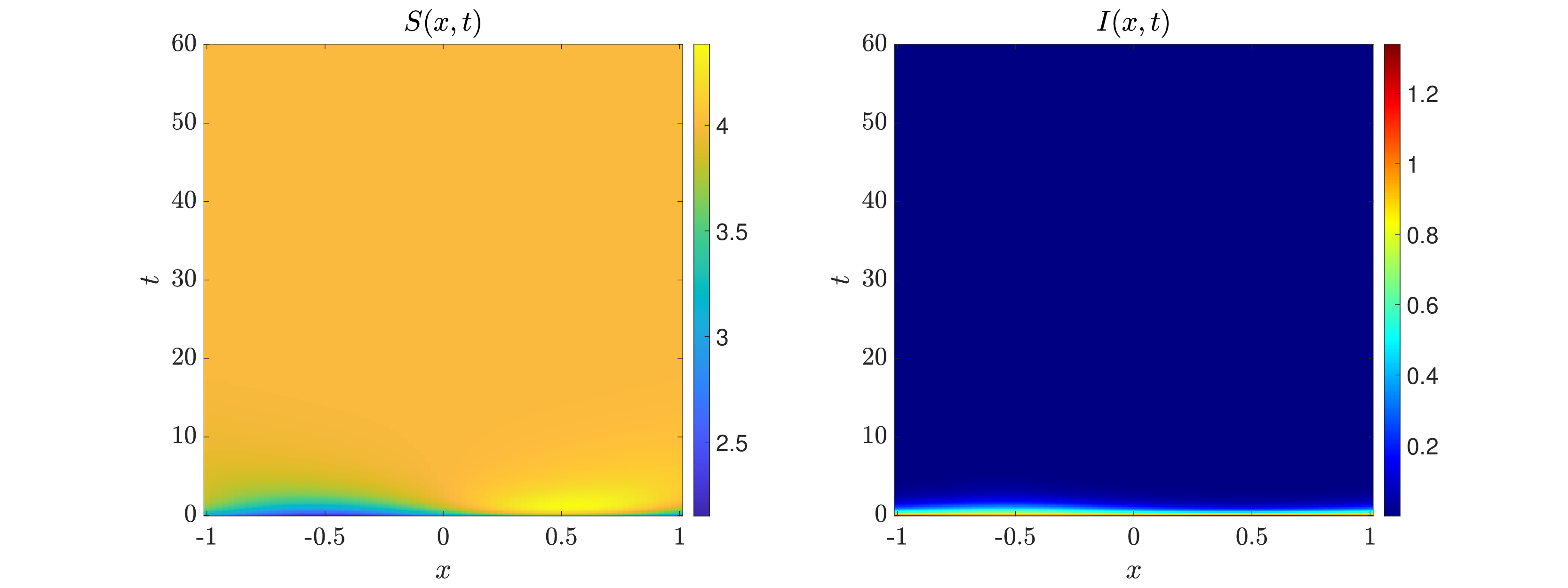}
		\centerline{\scriptsize{(b) $S_{0}(x)=2.8+0.65\sin(\pi x)$, $I_{0}(x)=1.2+0.14\sin(\pi x)$} }
	\end{minipage}
	\begin{minipage}[t]{1\linewidth}
		\centering
		\includegraphics[height=0.32\linewidth,width=0.99\linewidth]{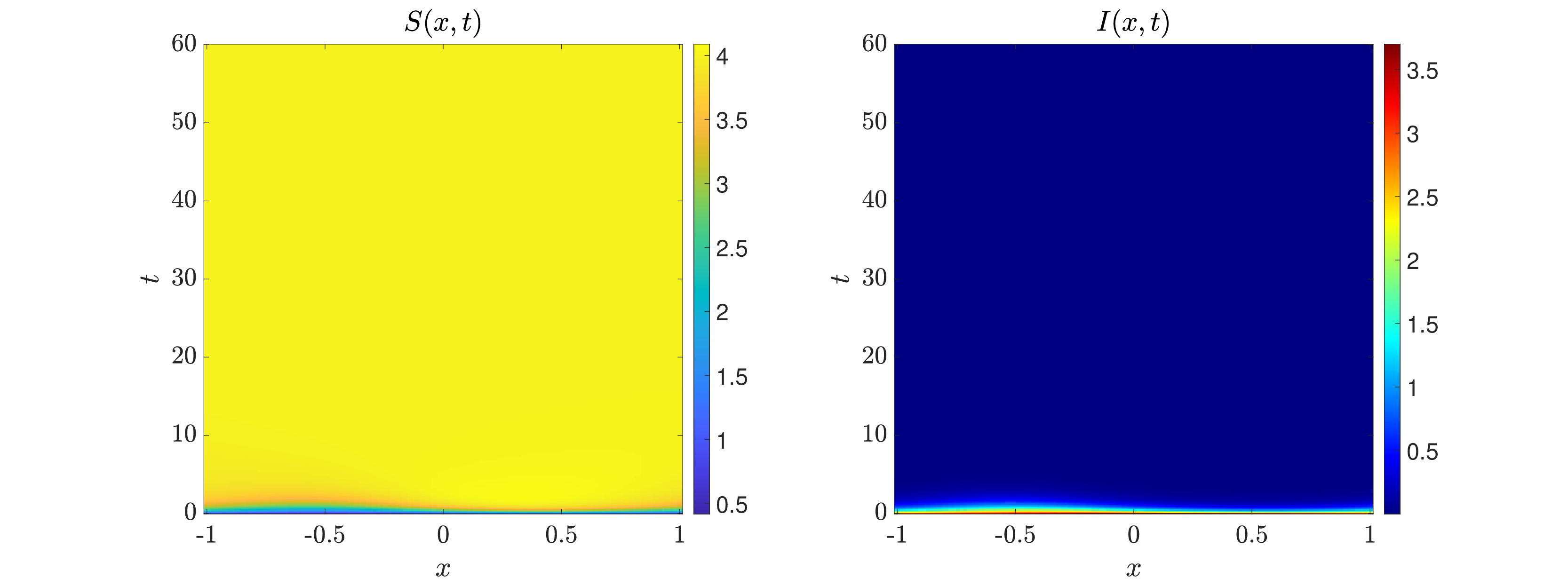}
		\centerline{\scriptsize{(c) $S_{0}(x)=0.6+0.18\sin(\pi x)$, $I_{0}(x)=3.4+0.3\cos(\pi x)$} }
	\end{minipage}
	\renewcommand{\figurename}{\footnotesize{\textbf{Fig.}}}
	\caption{\footnotesize{Spatiotemporal profiles of system (\ref{model1.2}) with $\beta(x,t)=\gamma(x,t)$ under three distinct spatially heterogeneous initial conditions. (a) Initial data 1, (b) Initial data 2, (c) Initial data 4.}}
	\label{Fig4}
\end{figure}

When $d_S=d_I=d$ and $R_0(d, N, m)\leq1$, the disease-free equilibrium \(\left(4, 0\right)\) is globally attractive on the basis of Theorem~\ref{theorem3.2}. To numerically verify this theoretical result, we first set $d_S = d_I = 1$ and select the values of $\beta$, $\gamma$, and $m$ from set~3 in Table~\ref{tab:params_six_sets}, with all other parameters and initial conditions identical to those in Fig.~\ref{Fig3}, which yields Fig.~\ref{Fig5}. A direct verification shows that $R_0 < 1$ under this parameter configuration. As can be seen from the figure, the solutions of system (\ref{model1.2}) converge to the disease-free equilibrium $(4, 0)$ for three distinct sets of initial data. We then choose parameter set~6 in Table~\ref{tab:params_six_sets} for $\beta$, $\gamma$, and $m$, while keeping all other parameters unchanged, giving $R_0 = 1$. To better illustrate the long-time dynamics of system (\ref{model1.2}) for different initial conditions, we fix several representative spatial points $x$ and present the corresponding time-series plots for four different sets of initial data in Fig.~\ref{Fig6}. The results show that, at these fixed spatial locations, the solutions of system (\ref{model1.2}) approach the disease-free equilibrium $(4, 0)$ as $t \to +\infty$.

\begin{figure}[htpp]
	\centering
	\begin{minipage}[t]{1\linewidth}
		\centering
		\includegraphics[height=0.32\linewidth,width=0.99\linewidth]{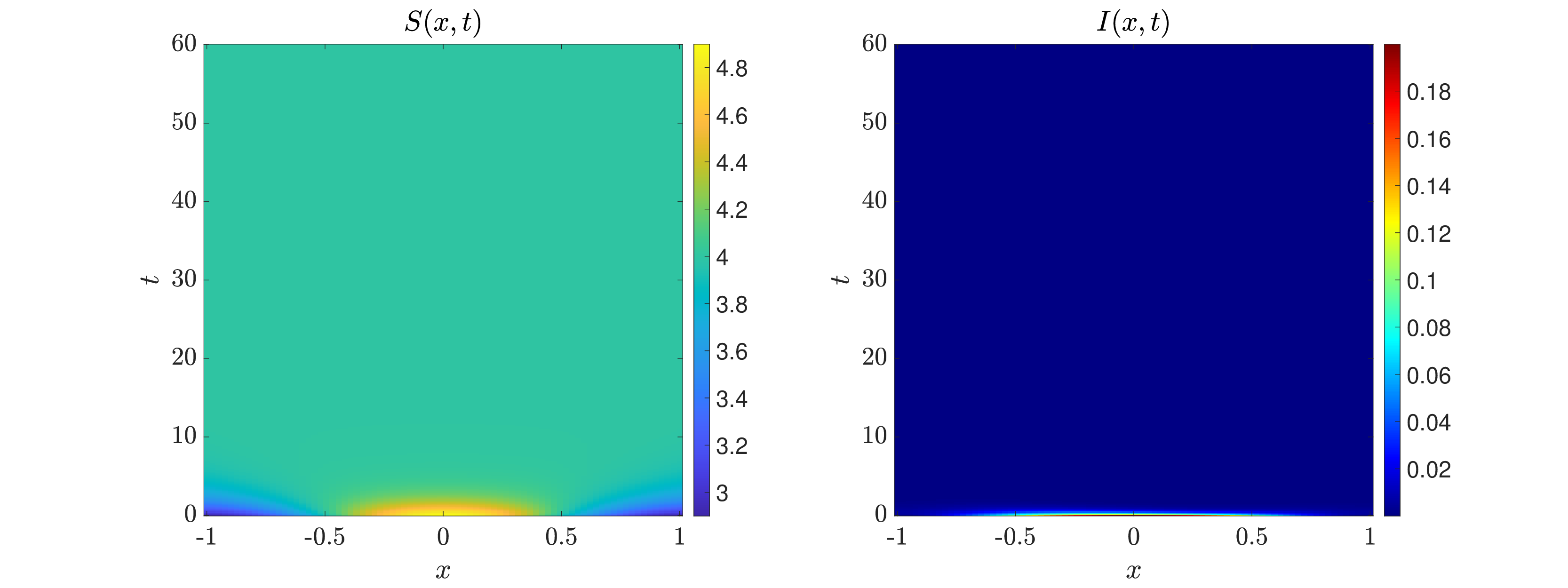}
		\centerline{\scriptsize{(a) $S_{0}(x)=3.9+\cos(\pi x)$, $I_{0}(x)=0.1+0.1\cos(\pi x)$} }
	\end{minipage}
	\begin{minipage}[t]{1\linewidth}
		\centering
		\includegraphics[height=0.32\linewidth,width=0.99\linewidth]{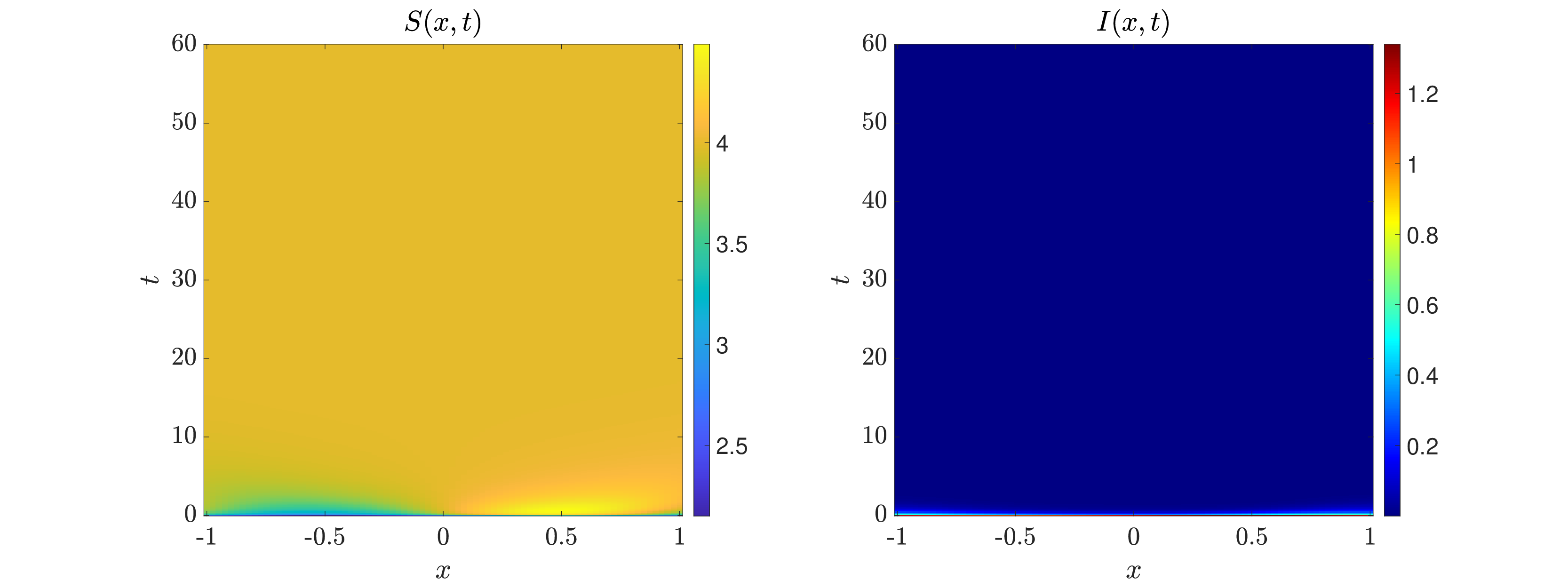}
		\centerline{\scriptsize{(b) $S_{0}(x)=2.8+0.65\sin(\pi x)$, $I_{0}(x)=1.2+0.14\sin(\pi x)$} }
	\end{minipage}
	\begin{minipage}[t]{1\linewidth}
		\centering
		\includegraphics[height=0.32\linewidth,width=0.99\linewidth]{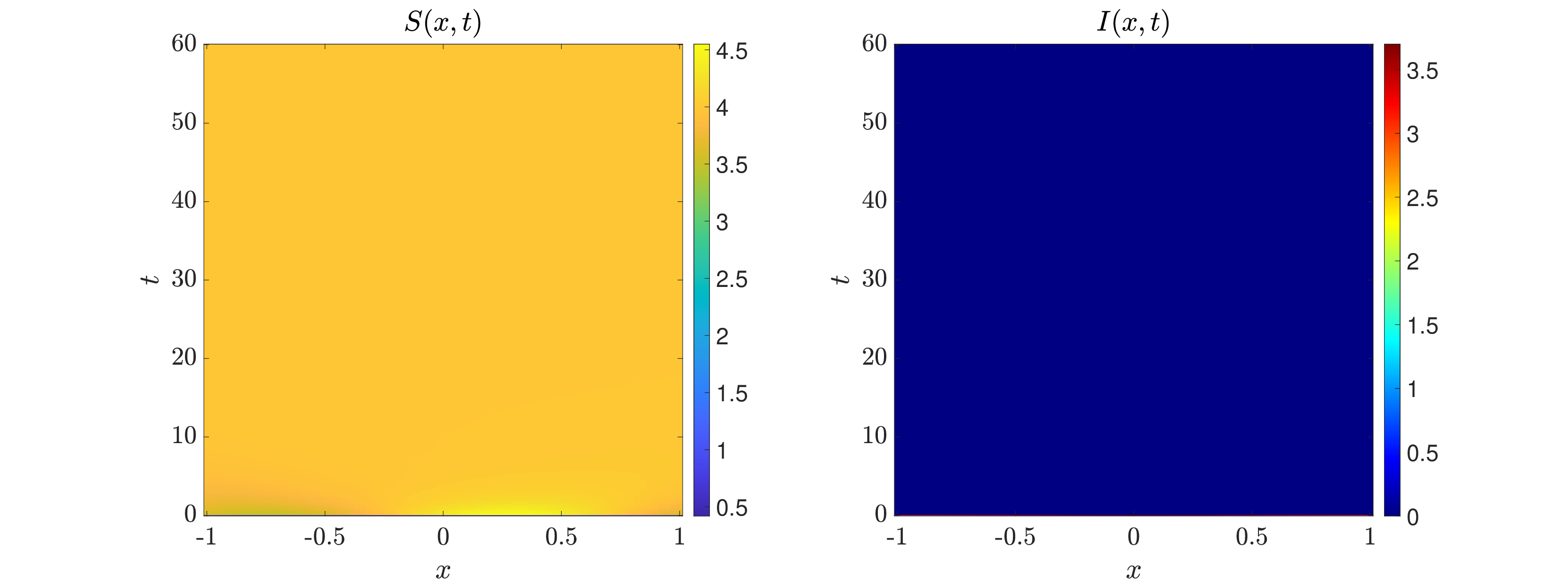}
		\centerline{\scriptsize{(c) $S_{0}(x)=0.6+0.18\sin(\pi x)$, $I_{0}(x)=3.4+0.3\cos(\pi x)$} }
	\end{minipage}
	\renewcommand{\figurename}{\footnotesize{\textbf{Fig.}}}
	\caption{\footnotesize{Spatiotemporal profiles of system (\ref{model1.2}) with $d_{S}=d_{I}=1$ and $R_0(d, N, m)<1$ under three distinct spatially heterogeneous initial conditions. (a) Initial data 1, (b) Initial data 2, (c) Initial data 4.}}
	\label{Fig5}
\end{figure}

\begin{figure}[htpp]
	\centering
	{
		\includegraphics[width=1.03\linewidth]{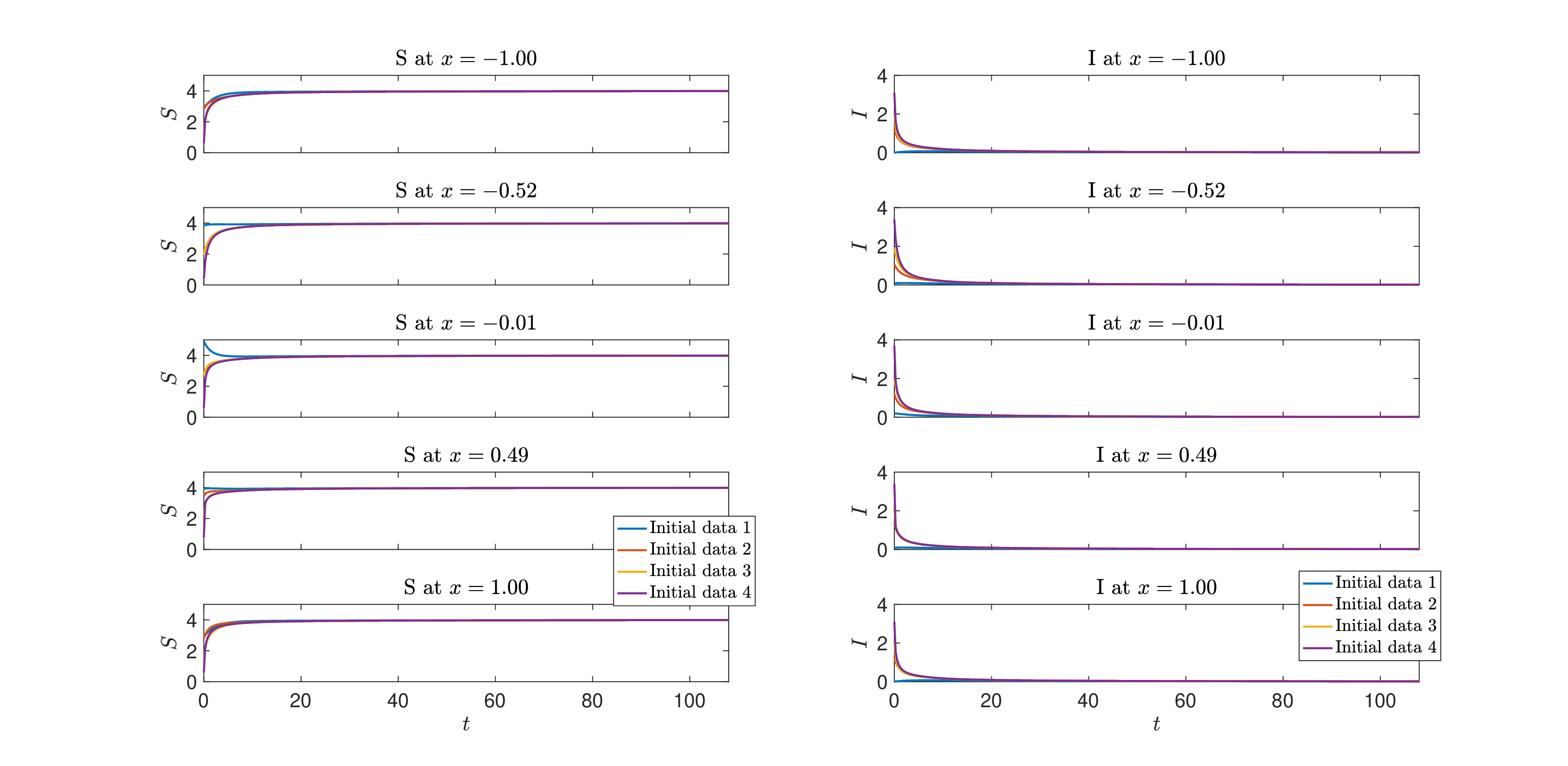}
	}
	\renewcommand{\figurename}{\footnotesize{\textbf{Fig.}}}
	\caption{\footnotesize{Time evolution of $S(x,t)$ and $I(x,t)$  with $d_{S}=d_{I}=1$ and $R_0(d, N, m)=1$ at fixed spatial positions $x=-1.00,-0.52,-0.01,0.49,1.00$ for four different initial conditions.}}
	\label{Fig6}
\end{figure}

With $d_S=d_I=1$, $N=8$ and common parameter values given in Table~\ref{tab:common_params}, we choose parameter set~1 in Table~\ref{tab:params_six_sets} for $\beta$, $\gamma$, and $m$. Direct calculation gives $R_0(d, N, m)>1$. Then there exists the unique endemic equilibrium $(S_{e}^*, I_{e}^*)$ by Theorem~\ref{theorem3.3}. Moreover, it is direct from Theorem~\ref{theorem3.4} that $(S_{e}^*, I_{e}^*)$ is globally attractive. As illustrated in Fig.~\ref{Fig7}, the solutions of system (\ref{model1.2}) corresponding to different initial data all converge to the same endemic equilibrium as $t \to +\infty$. To intuitively demonstrate that the solutions of system (\ref{model1.2}) converge to the same periodic solution starting from distinct initial data, we fix several representative spatial positions $x$, focus on the temporal evolution of $S(x,t)$ and $I(x,t)$, and plot Fig.~\ref{Fig8}.

\begin{figure}[htpp]
	\centering
	\begin{minipage}[t]{1\linewidth}
		\centering
		\includegraphics[height=0.32\linewidth,width=0.99\linewidth]{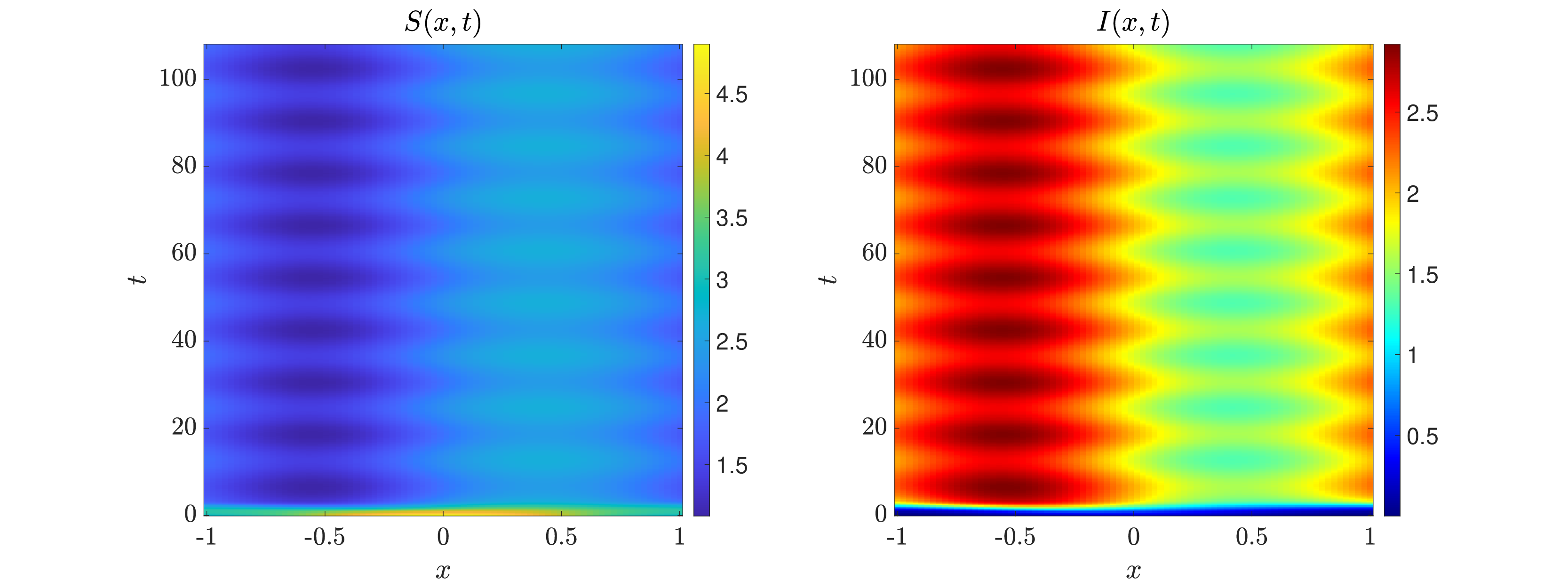}
		\centerline{\scriptsize{(a) $S_{0}(x)=3.9+\cos(\pi x)$, $I_{0}(x)=0.1+0.1\cos(\pi x)$} }
	\end{minipage}
	\begin{minipage}[t]{1\linewidth}
		\centering
		\includegraphics[height=0.32\linewidth,width=0.99\linewidth]{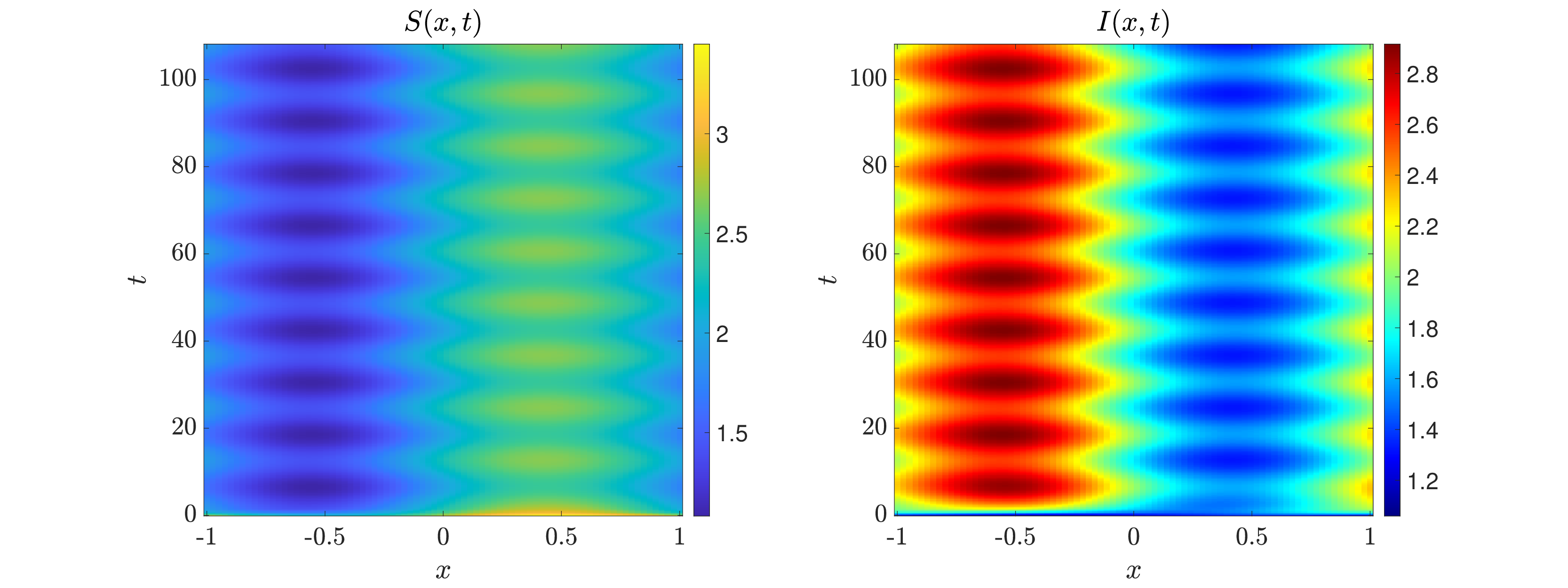}
		\centerline{\scriptsize{(b) $S_{0}(x)=2.8+0.65\sin(\pi x)$, $I_{0}(x)=1.2+0.14\sin(\pi x)$} }
	\end{minipage}
	\begin{minipage}[t]{1\linewidth}
		\centering
		\includegraphics[height=0.32\linewidth,width=0.99\linewidth]{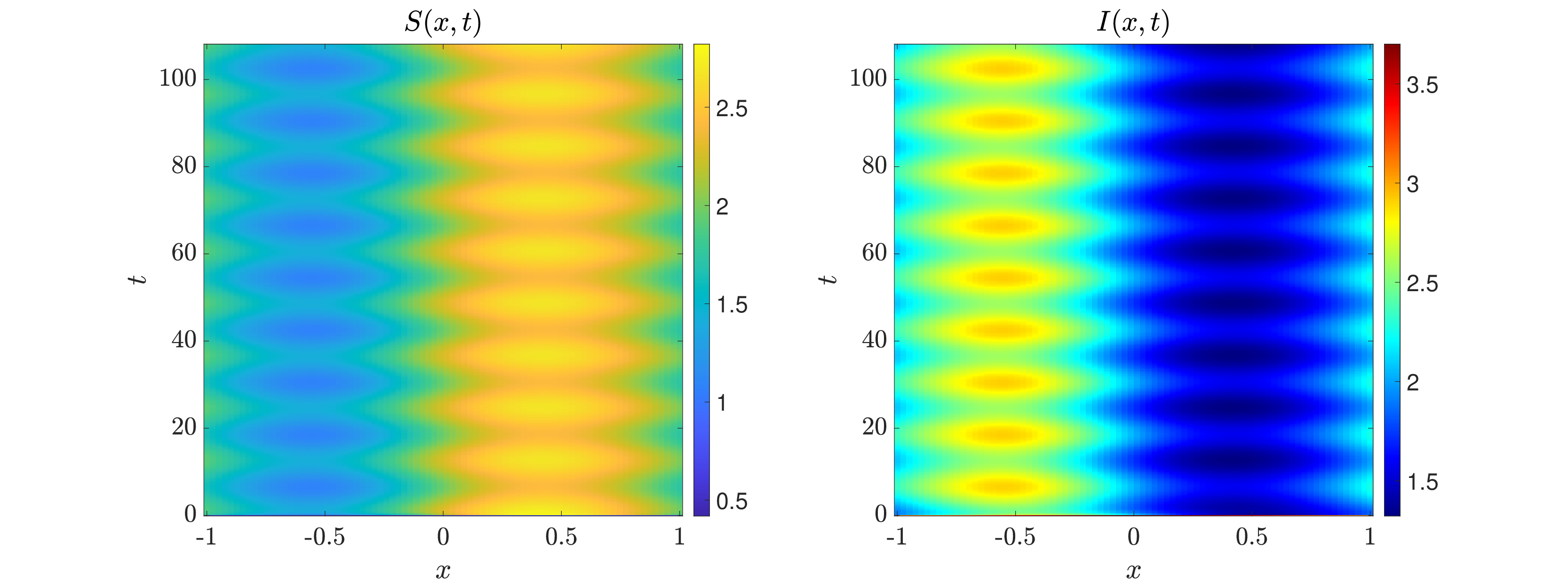}
		\centerline{\scriptsize{(c) $S_{0}(x)=0.6+0.18\sin(\pi x)$, $I_{0}(x)=3.4+0.3\cos(\pi x)$} }
	\end{minipage}
	\renewcommand{\figurename}{\footnotesize{\textbf{Fig.}}}
	\caption{\footnotesize{Spatiotemporal profiles of system (\ref{model1.2}) with $d_{S}=d_{I}=1$ and $R_0(d, N, m)>1$ under three distinct spatially heterogeneous initial conditions. (a) Initial data 1, (b) Initial data 2, (c) Initial data 4.}}
	\label{Fig7}
\end{figure}

\begin{figure}[h]
	\centering
	{
		\includegraphics[width=1.03\linewidth]{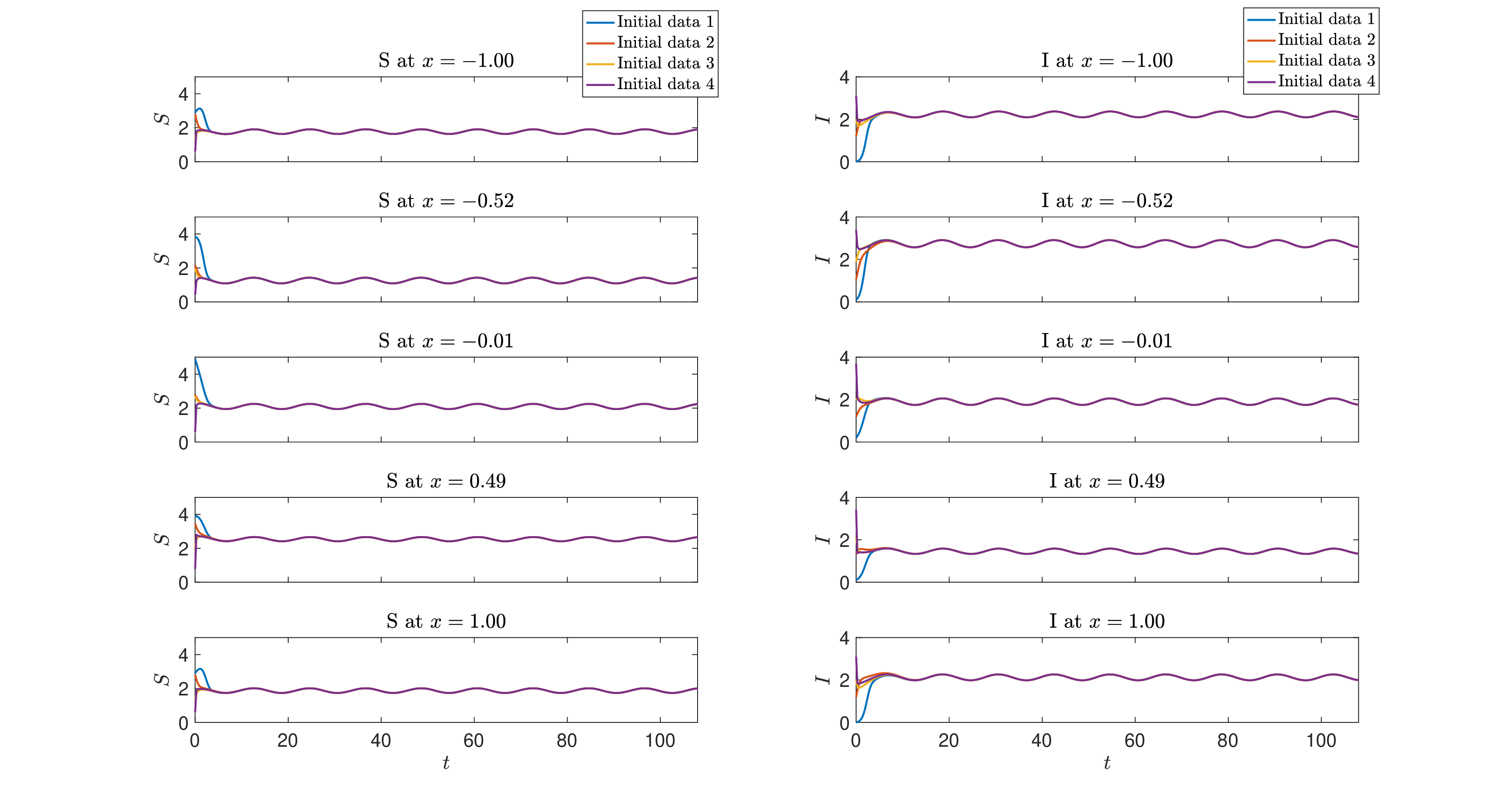}
	}
	\renewcommand{\figurename}{\footnotesize{\textbf{Fig.}}}
	\caption{\footnotesize{Time evolution of $S(x,t)$ and $I(x,t)$ with $d_{S}=d_{I}=1$ and $R_0(d, N, m)>1$ at fixed spatial positions $x=-1.00,-0.52,-0.01,0.49,1.00$ for four different initial conditions.}}
	\label{Fig8}
\end{figure}

\section{Conclusions and discussions}
	\label{section 6}

In this paper, we investigated a nonlocal dispersal SIS epidemic model with the saturated incidence function $\frac{SI}{m+S+I}$ in a spatially heterogeneous and temporally periodic environment. We first presented the basic reproduction number $R_0$ for system (\ref{model1.2}) and established its variational characterization. Furthermore, we examined the influence of dispersal rates, total population size and saturation parameters on the basic reproduction number. Then we derived the conditions for the existence and uniqueness of steady states for system (\ref{model1.2}). In addition, we established the global attractivity of both the disease-free and endemic equilibria for several special scenarios. We further studied how small saturation coefficients as well as low and high diffusion rates influence the profiles of the endemic equilibrium. Finally, we carried out numerical simulations.

It is noted that for SIS epidemic models with standard incidence \cite{Allen2008,Yang2019,Lin2023,Lin2024,Feng2025}, the basic reproduction number depends only on the diffusion coefficient $d_I$ of infected individuals. For SIS epidemic models with bilinear incidence \cite{Feng2022}, the basic reproduction number depends on both $d_I$ and the total population size $N$. However, the system considered in this paper incorporates the saturation effect, which makes $R_0$ depend on $d_I$, $N$, and the saturation parameter $m$, and also causes $N$ to have a significant impact on the disease dynamics. As the saturation parameter $m$ tends to zero, the basic reproduction number and the endemic equilibrium reduce to those of the standard-incidence model, respectively. The monotonicity of $R_0$ with respect to $d_I$ is difficult to obtain. Nevertheless, from Corollary \ref{corollary2.1} (iv)(v), it follows that a larger total population size enhances the transmission ability of the disease, while a stronger saturation effect weakens the transmission ability of the disease. Consequently, managing population density and reinforcing the saturation effect through interventions like contact restrictions and isolation policies are both beneficial for disease control.

To elucidate the biological implications of our analytical results, we adopt the terminology inspired by \cite{Allen2008,Peng2013,Peng2012}. A site is called a period-averaged low-risk (respectively, high-risk) site if its transmission rate averaged over one period is lower (respectively, higher) than its recovery rate averaged over the same period. Accordingly, $H^{-}$ and $H^{+}$ denote the sets of all period-averaged low-risk and high-risk sites, respectively. Combining Corollary \ref{corollary2.2}, Sections \ref{section 3} and \ref{section 4}, we obtain the  following biological interpretation.
	\begin{itemize}
	\item[{\rm (i)}]
If every site is period-averaged low-risk, then the disease eventually dies out, provided that susceptible and infected individuals have the same sufficiently small dispersal rate.
	
\item[{\rm (ii)}] If period-averaged high-risk sites exist, then sufficiently slow dispersal of infected individuals gives rise to a threshold population size. Furthermore, if susceptible and infected individuals have the same dispersal rate, the disease eventually dies out when the total population size is below the threshold, whereas it persists at an endemic level when the total population size exceeds the threshold.
\end{itemize}

\noindent Therefore, whether restricting the mobility of susceptible and infected individuals can eradicate the disease depends not only on the spatial distribution of the period-averaged transmission and recovery rates, but also on the total population size. This feature distinguishes our model from those with standard or bilinear incidence functions.

Additionally, there are some topics worth further investigations. For instance, whether $R_0$ is globally monotone with respect to $d_I$? Also, it should be pointed out that most of the conclusions in Sections~\ref{section 3} and~\ref{section 4} are based on the condition $d_S = d_I$. However, in reality, the dispersal rates of susceptible and infected individuals are often different. Therefore, it is necessary to explore the case where $d_S \neq d_I$, which would be interesting but may also give rise to more complex dynamical behaviors. These are left for future works.

\section*{Availability of data and material}
Not applicable.

\section*{Competing interests}
The authors declare that they have no competing interests.

\section*{Acknowledgments}
This work was supported by the National Natural Science Foundation of China (Nos. 12471176 and 12071491) and Guangdong Basic and Applied Basic Research Foundation (No. 2025A1515012221).  X. Lin was additionally supported by the China Postdoctoral Science Foundation (No. 2025M783159).



     \end{document}